\documentclass[12pt,a4paper]{article}

\usepackage{latexsym,amsfonts,amsmath,amsthm,amssymb,color}
\def\mcomment#1{}
\def\mcommentlast#1{}
\def\mcommentred#1{}
\def\jcomment#1{}

\newtheorem{theorem}{Theorem}
\newtheorem{lemma}{Lemma}
\newtheorem{corollary}{Corollary}
\newtheorem{proposition}{Proposition}

\newtheorem{definition}{Definition}
\newtheorem{algorithm}{Algorithm}
\newtheorem{remark}{Remark}

\newcommand{\satop}[2]{\stackrel{\scriptstyle{#1}}{\scriptstyle{#2}}}

\def\Hom{{{\mathrm {Hom}}}}
\def\bbK{{{\mathbb K}}}
\def\bbX{{{\mathbb X}}}
\def\bbN{{{\mathbb N}}}
\def\bbZ{{{\mathbb Z}}}
\def\bbC{{{\mathbb C}}}
\def\pr{{{\mathrm{pr}}}}

\allowdisplaybreaks

\begin{document}

\title{\scshape On the fast computation of the weight enumerator
polynomial and the $t$ value of digital nets over finite abelian groups}

\author{Josef Dick \\ School of Mathematics and Statistics,\\
University of New South Wales, \\
Sydney, NSW, 2052, Australia \\ \texttt{josef.dick@unsw.edu.au} \\ and \\ Makoto Matsumoto \\ Graduate School of Mathematical Science \\
University of Tokyo \\ 3-8-1 Komaba, Meguro-ku Tokyo 153-8914 \\ \texttt{matumoto@ms.u-tokyo.ac.jp} }

\date{\today}
\maketitle

\begin{abstract}
In this paper we introduce digital nets over finite abelian groups which contain digital nets over finite fields and certain rings as a special case. We prove a MacWilliams type identity for such digital nets. This identity can be used to compute the strict $t$-value of a digital net over finite abelian groups. If the digital net has $N$ points in the $s$ dimensional unit cube $[0,1]^s$, \jcomment{$\leftarrow$} then the $t$-value can be computed in $\mathcal{O}(N s \log  N)$ operations and the weight enumerator polynomial can be computed in $\mathcal{O}(N s (\log  N)^2)$ operations, where operations mean arithmetic of integers. By precomputing some values the number of operations of computing the weight enumerator polynomial can be reduced further.
\mcomment{ I erased all $b$ in the $\log (N)$ appeared inside $O()$. Anyway it is true. It is strange that if $b$ increases, then complexity goes down.}
\end{abstract}

\section{Introduction}
Digital nets are point sets
$\{\boldsymbol{x}_0,\ldots, \boldsymbol{x}_{N-1}\}$
in the $s$ dimensional unit cube $[0,1)^s$ with $N = b^m$ points (where $N, b \ge 2$ and $m \ge 1$ are integers),
whose construction is based on linear algebra over a finite field
(or more generally
finite rings) \cite{DP10, LNS, niesiam}.
\mcomment{I added \cite{LNS} and removed ``certain'' finite rings. Logically you are correct, since they deals with commutative associative with unit, but actually any finite ring will do.}
The aim of these constructions is to obtain highly uniformly distributed point sets. Such point sets are useful as quadrature points for quasi-Monte Carlo rules
\begin{equation*}
\frac{1}{N} \sum_{l =0}^{N-1} f(\boldsymbol{x}_l)
\end{equation*}
which are used to approximate integrals
$\int_{[0,1)^s} f(\boldsymbol{x}) \,\mathrm{d} \boldsymbol{x}$.
\mcomment{$[0,1]\to [0,1)$.}
The Koksma-Hlawka inequality states that the error is bounded by the variation of the integrand $f$ times the discrepancy of the quadrature points \cite[Chapter~2]{DP10}. One measure of the distribution properties of digital nets is the concept of $(t,m,s)$-nets in base $b$ \cite{DP10, niesiam}. A point set $P$ in $[0,1)^s$ consisting of $b^m$ points is a $(t,m,s)$-net in base $b$ if every interval of the form $$\prod_{i=1}^s \left[\frac{a_i}{b^{d_i}}, \frac{a_i+1}{b^{d_i}}\right)$$ for every choice of integers $0 \le a_i < b^{d_i}$ and integers $d_1,\ldots, d_s \ge 0$ with $d_1 + \cdots + d_s = m-t$, contains $b^t$ points. Thus the $t$-value measures the quality of $(t,m,s)$-nets (where smaller is better) \cite{DP10, niesiam}. If $t$ is the smallest integer such that the above property holds, then $P$ is a strict digital $(t,m,s)$-net and we call $t$ the exact $t$-value of $P$. \jcomment{$\leftarrow$}

Many explicit constructions of $(t,m,s)$-nets with good quality
parameter $t$ are known, see \cite{faure,nie86, niexi,sob} or \cite[Chapter~8]{DP10}.
These are based on the concept of digital nets which we recall
\mcomment{introduce $\to$ recall}
in the following. For a positive integer $b$,
$\bbZ_b=\{0,1,\ldots,b-1\}$ denotes the residue ring modulo $b$.
For a prime power $q$, $\mathbb{F}_q$ denotes the
$q$-element finite field.
\mcomment{$\leftarrow$ I added the definitions. May be moved to the beginning of this section.} 
Let $C_1,\ldots, C_s \in \mathbb{F}_q^{n \times m}$ be $n \times m$ matrices over a finite field $\mathbb{F}_q$ (with $n \ge m$). Let $\varphi:\mathbb{Z}_q \to \mathbb{F}_q$ be a bijection. For $0 \le l < q^m$ let
\begin{equation*}
l = \sum_{r=0}^{m-1} l_r q^r, \quad \mbox{where } l_r \in \mathbb{Z}_q,
\end{equation*}
be the base $q$ expansion of $l$. Let $\vec{l} = (\varphi(l_0), \ldots, \varphi(l_{m-1}))^\top \in \mathbb{F}_q^m$ and
\begin{equation*}
\vec{y}_{j,l} = C_j \vec{l}, \quad \mbox{for } 1 \le j \le s,
\end{equation*}
where $\vec{y}_{j,l} = (y_{j,l,1},\ldots, y_{j,l,n})^\top \in \mathbb{F}_q^n$. Then we define
\begin{equation*}
x_{l,j} = \frac{\varphi^{-1}(y_{j,l,1})}{q} + \cdots + \frac{\varphi^{-1}(y_{j,l,n})}{q^n}, \quad \mbox{for } 1 \le i \le s.
\end{equation*}
The point set $\boldsymbol{x}_l = (x_{l,1},\ldots, x_{l,s})$ for $0 \le l < q^m$ is called a digital net (over $\mathbb{F}_q$). A digital net which is a (strict) $(t,m,s)$-net is called a (strict) digital $(t,m,s)$-net.

Let $C_j = (\mathbf{c}_{j,1}^\top,\ldots, \mathbf{c}_{j,m}^\top)^\top$, i.e., $\mathbf{c}_{j,k}$ denotes the $k$th row of $C_j$. Then the condition that $C_1,\ldots, C_s$ generate a digital $(t,m,s)$-net over $\mathbb{F}_q$ is equivalent to the condition that for all nonnegative integers $d_1,\ldots, d_s$ with $d_1 + \cdots + d_s = m-t$, the vectors $$\mathbf{c}_{1,1},\ldots, \mathbf{c}_{1,d_1},\ldots, \mathbf{c}_{s,1},\ldots, \mathbf{c}_{s,d_s}$$ are linearly independent \cite[Theorem~4.52]{DP10}. \jcomment{$\leftarrow$ Added reference} Using this definition it is expensive to compute the exact $t$-value (i.e., the smallest value of $t$ for which the linear independence condition holds) since in general many linear independence conditions need to be verified. A direct computation of the $t$-value based on the linear independence properties of the generating matrices over a finite field is presented in \cite{PSch01}. In this paper we show how the exact $t$-value can be computed without checking any linear independence condition. This is done by using a Fourier inversion method which yields a fast algorithm for computing the exact $t$-value. Apart from the explicit constructions of digital nets, computer search algorithms of $(t,m,s)$-nets are also useful in that they often yield digital nets with very small $t$-value \cite{LNS,PSch01,Schm00}. The usefulness of such search algorithms is limited by the size of the search space of digital nets and the computational cost of computing the $t$-value of some given point set \cite{PSch01}.

In this paper we generalize the concept of digital nets over finite
fields \cite{DP10, niesiam} or \mcomment{removed ``certain''} finite
rings
\cite{LNS} to digital nets over a finite abelian group $G$.
We also generalize the concept of duality theory of nets to finite abelian groups as studied in \cite{np}
for finite fields. In our context, the dual net
\mcomment{(I replaced dual space with dual net.)}
is now defined via the
dual group of characters of the finite abelian group $G$.
We show how the quality parameter $t$ of a digital net is related to the
NRT
(Niederreiter~\cite{nie86} and Rosenbloom-Tsfasman~\cite{RT}) weight in the dual group of characters.

Further we introduce an algorithm which allows one to compute the weight
enumerator polynomial \cite{DS02, MS99, T08, vL92} of the dual of \mcomment{(added dual)} a given digital net in $\mathcal{O}(N s (\log  N)^2)$ operations (see Algorithm~\ref{alg1}) and the quality parameter $t$ of a given digital net in $\mathcal{O}(N s \log  N)$ operations (see Algorithm~\ref{alg2}), where operations always mean arithmetic of integers. By precomputing some values or using a faster polynomial multiplication algorithm this construction cost can (theoretically) be reduced further. For instance, if one wants to compute the $t$-value of many digital nets (as in a computer search algorithm which generates many digital nets and then chooses the best one, as in \cite{PSch01,Schm00}, or one wishes for instance to optimize the direction numbers of a Sobol' sequence as in \cite{KJ08}), it can be beneficial to store some values which might have to be computed repeatedly otherwise. It is also possible to balance the computational cost and storage cost in different ways, as will be clear from the result below.

The main idea for the fast computation of the $t$-value is to interpret the worst-case error of integration in the Walsh space \cite{DP05} in a different form. This yields another way of measuring the quality of a digital net (which is not based on the $t$-value). In this case a possible criterion is of the form
\begin{equation*}
-1 + \frac{1}{b^m} \sum_{l=0}^{b^m-1} \sum_{\boldsymbol{k} \in \{0,\ldots, b^m-1\}^s} \left(b^{-1}\right)^{\mu(\boldsymbol{k})} \mathrm{wal}_{\boldsymbol{k}}(\boldsymbol{x}_l),
\end{equation*}
where $\mu(\boldsymbol{k})$ is the NRT weight \jcomment{replaced 'some function' with 'NRT weight'} which measures the magnitude of $\boldsymbol{k}$ in some sense and $\mathrm{wal}_{\boldsymbol{k}}$ is the $\boldsymbol{k}$th Walsh function (for details see Section~\ref{sec:net-group}). In the following we replace $b^{-1}$ by a variable $z$, thus we consider
\begin{equation*}
-1 + \frac{1}{b^m} \sum_{l=0}^{b^m-1} \sum_{\boldsymbol{k} \in \{0,\ldots, b^m-1\}^s} z^{\mu(\boldsymbol{k})} \mathrm{wal}_{\boldsymbol{k}}(\boldsymbol{x}_l).
\end{equation*}
We view this expression as a polynomial in the variable $z$, i.e., we do not substitute any value for $z$, but rather aim at expressing this polynomial in the form
\begin{equation*}
\sum_{a=1}^{sm} z^a N_a.
\end{equation*}
This polynomial is shown to be the weight enumerator polynomial of
{\em the dual net} $\mathcal{D}$, see \cite[Definition~4.76]{DP10} or Section~\ref{sec:net-group} for details,
namely we have a MacWilliams type identity
\begin{equation*}
\frac{1}{b^m} \sum_{l=0}^{b^m-1} \sum_{\boldsymbol{k} \in \{0,\ldots, b^m-1\}^s} z^{\mu(\boldsymbol{k})} \mathrm{wal}_{\boldsymbol{k}}(\boldsymbol{x}_l) = \sum_{a=0}^{ms} z^a N_a.
\end{equation*}
Different but similar identities have previously been studied in coding theory, association schemes, orthogonal arrays, digital nets and so on \cite{DS02, MS99, T08, vL92}. In particular, a similar approach was considered in \cite{T08}, but with a different aim.

Here the aim is to compute the coefficients $N_a$ for $1 \le a \le m$. We show that these coefficients can be computed in $\mathcal{O}(N s (\log  N))$ operations. On the other hand, since these coefficients are related to the $t$-value of the digital net, we can also compute the $t$-value of a digital net in $\mathcal{O}(N s (\log  N))$ operations (see Algorithm~\ref{alg1}).

A second result uses a MacWilliams identity in the reverse direction
(see Algorithm~\ref{alg2}). \mcomment{I added alg2.}
This simplifies the computation of the $t$-value,
but it does not yield the weight enumerator polynomial.
The construction cost of this algorithm is $\mathcal{O}(N s \log  N)$
operations.
\mcomment{I commented out the comparison of two algorithms, since it is
below.}

We present some background in the following
section. Section~\ref{sec:net-group}
introduces the notion of $(t,m,s)$-nets over {\em finite abelian groups},
\mcomment{($\leftarrow$)}
and deals with the proof of the main
results and gives some discussion of implementing the algorithm.  We
also include some results if the point set is not a digital net. In Section~\ref{sec:num-res}, as a proof of concept, we use our methods
to compute the exact $t$-values of Sobol's sequence \cite{sob} (as implemented in
Matlab 2011a) for dimensions $3 \le s \le 22$ and number of points $2^m$
with $2 \le m \le 25$. \jcomment{Corrected the range of $s$ and $m$ according to the result in the table} In our experiments, Algorithm~\ref{alg2} is
slightly faster than Algorithm~\ref{alg1} for computing the
$t$-value (we did not use precomputation for both algorithms). \jcomment{Changed 'significantly' to 'slightly'}
\mcomment{I moved the last sentence from above.}
In Section~\ref{sec:group-ring} we introduce a further
\mcomment{(I added further)}
generalization of $(t,m,s)$-nets and digital nets, over finite abelian
groups. We also study for which finite rings the ring-theoretic dual net
coincides with the character-theoretic dual net used for digital nets over finite abelian groups.


\section{Digital nets and Walsh functions over finite abelian groups}
\label{sec:net-group}

Let $(G,+)$ be a finite abelian group of $b$ elements. Let $T = \{z \in \mathbb{C}: |z|=1\}$ be the multiplicative group of complex numbers
of absolute value one.\index{}
\begin{definition} Let $G$ be a finite abelian group.
The dual group of $G$ is defined by $G^*:=\Hom(G,T)$,
namely the set of group homomorphisms from $G$ to $T$,
often called {\em characters} of $G$.
\end{definition}
The following results are well-known~\cite[Part I]{Serre}. Let $G, G_1, G_2$ be
finite abelian groups.
\begin{lemma}\label{lem:character}
\begin{enumerate}
\item \label{exponent}
Let $e$ be the exponent of $G$, namely, the maximum order of
elements in $G$. Then, the image of $\Hom(G,T)$ is exactly
the cyclic group $<\zeta_e>$ of order $e$, where
$\zeta_e:=\exp(2 \pi \mathrm{i}/e)$.
\item \label{isomorphism}
$G^*$ is isomorphic to $G$ as a finite abelian group,
but there is no canonical choice of the isomorphism.
\item \label{dual-map}
A group homomorphism $f:G_1 \to G_2$ induces
$f^*:G_2^* \to G_1^*$ by composition $k \in G_2^* \mapsto k\circ f \in
      G_1^*$.
\item \label{involution}
$G \to (G^*)^*$, $x \mapsto (x^{**} : k \mapsto k(x))$
gives an isomorphism, through which we identify $G=(G^*)^*$.

\item \label{surjectivity}
A morphism $f$ is surjective (respectively injective)
if and only if $f^*$ is injective (respectively surjective).

\item \label{product}
The dual of $G_1 \times G_2$ is canonically
isomorphic to $G_1^* \times G_2^*$.

\item \label{orthogonality}
For any $k \in G^*$, the sum $\sum_{x \in G}k(x)$
is $0$ if $k \neq 0$, and $\#(G)$ if $k = 0$.
Dually, for any $x \in G$,
the sum $\sum_{k \in G^*}k(x)$ is $0$ if $x \neq 0$,
and $\#(G)$ if $x = 0$.
\end{enumerate}
\end{lemma}
\begin{remark}
The notion of dual $G^*$ generalizes to
locally compact Hausdorff abelian groups $G$
by defining $G^*:=\Hom_{cont}(G, T)$ (namely,
the continuity is required for the characters), and
the above lemma holds for this wider class (with some small adjustments),
where finite abelian groups have discrete topology.
See \cite{Pontryagin} or \cite[Chapter VII]{loomis}
for the basic of such groups and Fourier transformation on them.
\end{remark}
\mcomment{Indices of $\kappa, \xi$ may start with 0 or 1; I made
everything into 1.}

We return to our finite abelian group $G$.
Define the direct product of denumerably many copies of $(G,+)$
and the direct sum of denumerably many copies of $(G^*,+)$
by:
\begin{eqnarray*}
\bbX &:=& G^{\bbN}
= \{ \mathbf{x} = (\xi_1, \xi_2,\ldots ) : \xi_i \in G \}
\mbox{ and } \\
\bbK &:=&
\{ \mathbf{k} = (\kappa_1, \kappa_2,\ldots ) \in (G^*)^{\mathbb{N}}: \kappa_i = 0 \mbox{ for almost all } i\}.
\end{eqnarray*}
We remark that $\bbK$ and $\bbX$ are dual to each other,
through the pairing $\bullet$:
$$
\bbK \times \bbX \to T,
\quad (\kappa_1, \kappa_2,\ldots )\bullet (\xi_1, \xi_2, \dots):=
\prod_{i=1}^\infty \kappa_i(\xi_i) \in T.
$$

For a positive integer $n$, we have a truncated version
$$
\bbK_n:=(G^*)^n, \bbX_n:=G^n,
(\kappa_1, \ldots, \kappa_n)\bullet (\xi_1,\ldots \xi_n)
:=\prod_{i=1}^n \kappa_i(\xi_i) \in T.
$$
These finite groups are again dual to each other.

We fix bijections $\varphi:\mathbb{Z}_b \to G$
and $\phi: \mathbb{Z}_b \to G^*$
with $\varphi(0) =0$ and $\phi(0) =0$.
\mcommentlast{I noticed this is necessary to make the bijection below.}
The bijection $\phi$ gives a bijection
denoted by the same symbol by an abuse of notation:
$$
\phi:\bbN_0 \to \bbK
$$
in the following way:
We identify elements in $\bbK$ with a nonnegative integer $k\in \mathbb{N}_0$
with $b$-adic expansion
$k =
\kappa_1 + \kappa_2 b + \cdots + \kappa_{m} b^{m-1}
$
with $\kappa_i \in \bbZ_b$.
Then we define
$
\phi(k):=(\phi(\kappa_1), \ldots, \phi(\kappa_{m}), 0, \ldots) \in
\bbK
$, which is a bijection.
By restriction, we have a truncated version
$$
\phi: \{0,1,\ldots, b^{m}-1\} \to \bbK_m.
$$
Analogously
we use the same symbol $\varphi$ to denote
$$
\varphi: [0,1) \to \bbX
$$
defined as follows. For a number $x \in [0,1)$ with base $b$
expansion $x = \xi_1 b^{-1} + \xi_2 b^{-2} + \cdots$,
we define
$\varphi(x):= (\varphi(\xi_1), \varphi(\xi_2),\ldots) \in \bbX$.
Here we use the finite expansion of $x$ if $x$ is a $b$-adic rational.
Because of this reason
$\varphi: [0,1) \to \bbX$ is not surjective.
However, both have natural measures
(namely, Lebesgue measure and Haar measure)
and $\varphi$ preserves the measure,
through which $[0,1)$ and $\bbX$ are
isomorphic as a measured space when a
measure-zero subset is removed from $\bbX$.
There is also a continuous mapping
$\varphi^{-1}:\bbX \to [0,1]$
\mcomment{($[0,1) \to [0,1]$)}
given by
$
(\xi_1, \xi_2, \ldots)
\mapsto
\varphi^{-1}(\xi_1) b^{-1} + \varphi^{-1}(\xi_2) b^{-2} + \cdots,
$
which is surjective
(and inverse to $\varphi$ if measure zero sets are neglected).
\mcomment{$1$ should be removed from $[0,1]$, and the measure zero set
of $\bbX$ should be removed, hence plural.}

The Walsh functions are then defined the following way: for
$k = \kappa_1 + \kappa_2 b + \cdots + \kappa_{m} b^{m-1}$
 and $x = \xi_1 b^{-1} + \xi_2 b^{-2} + \cdots$ we have
\begin{equation*}
{}_{G}\mathrm{wal}_k(x) = \phi(k)\bullet \varphi(x)=
\prod_{i=1}^m \phi(\kappa_{i})(\varphi(\xi_i)).
\end{equation*}
For the $s$-dimensional case, we again use the \jcomment{$\leftarrow$ added 'the'} same letters
$$
\phi:\bbN_0^s \to \bbK^s, \quad
\varphi:[0,1)^s \to \bbX^s.
$$
For $\boldsymbol{k}=(k_1,\ldots k_s)$
and $\boldsymbol{x}=(x_1,\ldots,x_s)$
we define
$$
{}_{G}\mathrm{wal}_{\boldsymbol{k}}(\boldsymbol{x}) = \prod_{i=1}^s
{}_{G}\mathrm{wal}_{k_i}(x_i)
=\phi(\boldsymbol{k})\bullet \varphi(\boldsymbol{x})
$$


In the following we recall the definition of $(t,m,s)$-nets in base $b$
\cite{DP10, niesiam}. \mcomment{$(\leftarrow)$}
\begin{definition}\label{def:tms-net}
Let $b \ge 2$ and $s \ge 1$ be integers. A point set $P = \{\boldsymbol{x}_0,\ldots, \boldsymbol{x}_{b^m-1}\} \subset [0,1)^s$ is called a $(t,m,s)$-net in base $b$ if for all nonnegative integers $d_1,\ldots, d_s$ with $d_1 + \cdots + d_s = m-t$ the elementary intervals
\begin{equation*}
\prod_{i=1}^s \left[\frac{a_i}{b^{d_i}}, \frac{a_i+1}{b^{d_i}} \right)
\end{equation*}
contain exactly $b^t$ points for all choices of $0 \le a_i < b^{d_i}$ for $1 \le i \le s$.

If $t$ is the smallest value such that $P$ is a $(t,m,s)$-net, then we call $P$ a strict $(t,m,s)$-net and $t$ the exact quality parameter (or $t$-value). \jcomment{$\leftarrow$ added statement about 'exact quality parameter'}
\end{definition}
Because of this definition, we may replace the coordinates of the points
to be $m$-digit $b$-adic rationals.
We fix an $n \geq m$, and we consider point sets with coordinates being
all $n$-digit $b$-adic rationals. Namely,
we consider only the image of
$\varphi_n^{-1}: \bbX_n^s \to [0,1)^s.$ (For estimating the computational complexity we use $n = m$ to avoid unnecessary complications.)

We identify  $\bbX_n^s=(G^n)^s$ with the set of
$s\times n$ matrices $(M_{s,n}(G),+)$,
and similarly identify $\bbK_n^s=((G^{*})^n)^s=(M_{s,n}(G^*),+)$. For
$X \in \bbX_n^s$ we write
$X = (\xi_{i,j})_{\satop{1 \le i \le s}{1 \le j \le n}}$ and similarly
$K = (\kappa_{i,j})_{\satop{1 \le i \le s}{1 \le j \le n}}$.
\begin{definition}\label{def:combinatorial-net}
A subset $P \subset \bbX_n^s$ is a combinatorial $(t,m,s)$-net in base $b$
if $\varphi^{-1}(P)$ is a $(t,m,s)$-net in base $b$.
\end{definition}
Note that this definition does not use the group structure of $G$.
Note also that to define $(t,m,s)$-nets it is
enough to consider the case $n=m$, but for the convenience
in future use we choose to have general $n$.
\begin{definition}
Let $P \subset \bbX_n^s$ be a sub abelian group.
We define $P^\perp \subset \bbK_n^s$ as the kernel of
$(\bbK_n)^s = (\bbX_n^s)^* \to P^*$.
In other words,
$$
P^\perp:=
\{K \in \bbK_n^s \ | \ K \bullet X = 1
\mbox{ for any } X \in P
\}.
$$

We call $P^\perp$ the character theoretic dual net
of $P$.
\end{definition}

Let $Y$ be a subset of $\{1,2,\ldots,s\}\times \{1,2,\ldots,n\}$.
Let $\bbX^s_n(Y)$ be the direct product of $\#(Y)$ copies of $G$ with index
in $Y$,
namely:
$$
\bbX_n^s(Y):=\{(\xi_{i,j})| (i,j) \in Y, \xi_{i,j} \in G\}.
$$
We have an obvious projection
$$
\pr_Y: \bbX_n^s \to \bbX_n^s(Y), (\xi_{i,j})_{i,j} \mapsto (\xi_{i,j})_{(i,j)\in Y}.
$$

Let $P$ be a subset of $\bbX_n^s$, with inclusion map $\iota:P \to \bbX_n^s$.
%
A function $f:S_1 \to S_2$ is said to be {\em uniform}
if $\#(f^{-1}(s_2))$ is constant for any $s_2 \in S_2$.
In particular, it implies surjectivity if $S_1$ is non-empty. \mcommentlast{I added.}
Conversely, if $S_1$ and $S_2$ are finite groups
and $f$ is a group homomorphism, then
the surjectivity of $f$
implies uniformity, because
any element $s_1$ in the set $f^{-1}(s_2)$ gives
a bijection $f^{-1}(0) \to f^{-1}(s_2)$ obtained by adding $s_1$.

The point set $P$ is {\em $Y$-uniform} if $\pr_Y \circ \iota:P \to \bbX_n^s(Y)$
is uniform.

\begin{lemma}\label{lem:net-by-uniformness}
Fix a positive integer $m$. A finite subset $P\subset \bbX_n^s$ of
cardinality $b^m$ is a (combinatorial)
$(t,m,s)$-net if and only if the following conditions
are satisfied:

For any integers $d_1,\ldots,d_s\geq 0$,
let $Y:=Y(d_1,\ldots,d_s)$ denote the union of
$$\{(1,1), (1,2),\ldots,(1,d_1)\},
\{(2,1), (2,2),\ldots,(2,d_2)\},
\ldots,
\{(s,1), (s,2),\ldots,(s,d_s)\}.$$
Then, for any $d_1+d_2+\cdots+d_s \leq m-t$,
the composition
$
P \stackrel{\iota}{\hookrightarrow}
\bbX_n^s \stackrel{\pr_Y}{\to}
\bbX_n^s(Y)
$ is uniform, namely, $P$ is $Y$-uniform.
\end{lemma}
\begin{proof}
This follows from the definition of $(t,m,s)$-nets.
\end{proof}

The surjective homomorphism $\bbX_n^s \to \bbX_n^s(Y)$ induces
an injective homomorphism
$(\bbX_n^s(Y))^* \to (\bbX_n^{s})^*$. Through the identification
$(\bbX_n^{s})^* = \bbK_n^s=M_{s,n}(G^*)$,
$(\bbX_n^s(Y))^*\subset \bbK_n^s$ is identified with
$$
(\bbX_n^s(Y))^*=\{K=(\kappa_{i,j}) | \kappa_{i,j}=0 \mbox{ for all $(i,j) \notin Y$} \}
\subset M_{s,n}(G^*).
$$

The following easy lemma links \jcomment{replaced 'interprets' with 'links'} the $(t,m,s)$-net property of
$P$ to the minimum Niederreiter-Rosenbloom-Tsfasman (NRT) weight
(see, for example, \cite[Section~7.1]{DP10}) of $P^\perp$.
\begin{lemma}\label{lem:duality-in-perp}
Let $P\subset \bbX_n^s$ be a sub group.
For any subset $Y \subset \{1,2,\ldots,s\}\times \{1,2,\ldots,n\}$,
the composition $P \subset \bbX_n^s \to \bbX_n^s(Y)$ is surjective (namely $P$ is $Y$-uniform)
if and only if $P^\perp \cap (\bbX_n^s(Y))^*=\{0\}$.
In other words, if and only if any $K \in P^\perp$
with $\kappa_{i,j}=0$ for all $(i,j)\notin Y$ is $0$.
\mcomment{Original condition seems wrong, I replaced.
Probably, this ``if and only if'' condition may be removed.}
\end{lemma}
\begin{proof}
By \ref{surjectivity} in Lemma~\ref{lem:character},
the surjectivity is equivalent to the injectivity of
$\bbX_n^s(Y)^* \to P^*$, which is equivalent to
$\bbX_n^s(Y)^* \cap \mathrm{Ker}(\bbK_n^s \to P^*)=\{0\}$, and the
result follows from the definition of $P^\perp=\mathrm{Ker}(\bbK_n^s \to P^*)$.
\end{proof}
For $k =(\kappa_1,\ldots,\kappa_n)\in (G^*)^n=\bbK_n$ which
is not 0, we define
$$
\mu(k)=\max\{j | 1\leq j \leq n, \kappa_j \neq 0 \}
$$
and for $k=0$ we define $\mu(0):=0$.
For any $K \in (\bbX_n^s)^*=M_{s,n}(G^*)$, let
$k_i$ be the $i$-th row of $K$, and define
\begin{equation}\label{eq:NRT}
\mu(K) = \sum_{i=1}^s \mu(k_i),
\end{equation}
which is nothing but the definition of the NRT-weight of $K$.
For a subgroup $Q\subset \bbK_n^s$, we define
its minimum NRT-weight
\begin{equation}\label{eq:minNRT}
\mbox{minNRT}(Q)
:=\min\{
 \mu (K) | K \in Q, K \neq 0
\}.
\end{equation}
If $Q=\{0\}$, we define $\mbox{minNRT}(Q)= ns+1$.
\mcomment{Originally $(m+1)$, but I changed into $ns+1$.
 In $(t,m,s)$-net case, this occurs only if $s=1$ and $n=m$,
 which returns to the original one.
}
A direct generalization of the duality theorem for digital nets in \cite{np} is the following:
\begin{theorem}\label{th:dual-net-minwt}
Let $P\subset \bbX_n^s$ be a subgroup of cardinality $b^m$.
Then, $P$ is a $(t,m,s)$-net if and only if
$\mbox{minNRT}(P^\perp)\geq m-t+1$.
In other words, $m+1-\mbox{minNRT}(P^\perp)$
gives the exact value of $t$ for $P$, i.e., $P$ is a strict $(t,m,s)$-net.
\end{theorem}
\begin{proof}
By Lemma~\ref{lem:net-by-uniformness},
the $(t,m,s)$ condition is equivalent to
the surjectivity of $\pr_Y\circ \iota$ (since uniformness
is equivalent to the surjectivity for group homomorphisms),
and by Lemma~\ref{lem:duality-in-perp},
the condition is equivalent to that
$P^\perp \cap \bbX_n^s(Y(d_1,\ldots,d_s))^*=\{0\}$
holds for every $d_1+\cdots +d_s\leq m-t$, which is
equivalent to that any $K\neq 0$ with $\mu(K)\leq m-t$
does not belong to $P^\perp$, hence the condition of minimum weight of $P^\perp$.
\end{proof}

\subsection{A MacWilliams identity for digital nets over abelian groups}

In the following we define the weight enumerator polynomial.

\begin{definition}
We define the weight enumerator polynomial $\mathrm{WP}_{P^\perp}(z)$ of $P^\perp$
with respect to the NRT-weight, which is a non-negative integer
coefficient polynomial in variable $z$, by
$$
\mathrm{WP}_{P^\perp}(z) = \sum_{K \in P^\perp} z^{\mu(K)}
=
\sum_{a=0}^{ns} N_a z^{a}. \mcomment{\mbox{not $(n+1)s$}}
$$
\end{definition}
This gives definitions of $N_a$, i.e.,
$N_a$ is the number of matrices in $P^\perp$ with NRT-weight $a$:
$$
N_a = \# \{K \in P^\perp | \mu(K) = a\}.
$$
\mcomment{In the above, I used $\#$. Below, I change many $m$ to $ns$.
 Because, we treat general $P$ which may not be of cardinality
 $b^m$.
}
Thus we have
\begin{equation*}
\mbox{minNRT}(P^\perp) = \min \{a | 1 \le a \le ns, N_a \neq 0\}
\end{equation*}
and if $N_1 = \cdots = N_{ns} = 0$ we have
$\mbox{minNRT}(P^\perp) = ns+1$ (this means $P^\perp=\{0\}$).

We introduce the notion of dual NRT-weight. \mcomment{May have a better name.}
\begin{definition}\label{def:dual-NRT}
We shall define {\em the dual NRT-weight} $\mu^*$ on $\bbX_n$:
for $x=(\xi_1,\ldots,\xi_n) \in \bbX_n=G^n$, we define
$\mu^*(x)$ as the index of the minimum nonzero coordinate, i.e., if $x \neq 0$ we set
$$
\mu^*(x) = \min \{i| \xi_i \neq 0\}.
$$
If $x=0$, then $\mu^*(x):=0$.
\end{definition}
It is natural to define $\mu^*(X)$ for
$X=(x_1,\ldots,x_s)^\top \in \bbX_n^s$ by
\begin{equation*}
\mu^*(X):= \sum_{i=1}^s \mu^*(x_i),
\end{equation*}
but we do not use this later.
\mcomment{Art and I want to use this notion for a completely different
purpose. This might be not natural, think more???}

The following gives an algorithm to compute
the weight enumerator polynomial of $P^\perp$
by enumeration of $P$.

\begin{theorem}\label{th:MacWilliams}
Let $P$ be a subgroup of $\bbX_n^s$.
For $X \in \bbX_n^s$, let $x_i \in \bbX_n$
be its $i$-th row, namely,
$X=(x_1,\ldots,x_s)^\top \in \bbX_n^s$.
Then, the weight enumerator polynomial of $P^\perp$ is
given by
$$
\sum_{a=0}^{ns} N_a z^{a} =
\frac{1}{\#(P)} \sum_{X \in P} \prod_{i=1}^s p(\mu^*(x_i);z),
$$
where the polynomial $p(h;z)$ ($h$ being an integer between $0$ and $n$)
is defined by
$$
p(h;z)
:=
\left\{\begin{array}{rl} 1 + \sum_{a=1}^n z^a (b^a-b^{a-1})
& \mbox{if }
 h = 0,
 \\ 1 + \sum_{a=1}^{h-1} z^a (b^a-b^{a-1}) - z^hb^{h-1}
 & \mbox{if } h > 0. \end{array} \right.
$$
\mcomment{The above holds for any subgroup $P$.}
Suppose that $\#(P)=b^m$.
Then, $P$ is a strict $(t,m,s)$-net with
$$
t = m+1- \min \{a | 1 \le a \le ns, N_a \neq 0\},
$$
where for $N_1 = \cdots = N_{ns} = 0$
we set $\min \{a | 1 \le a \le ns, N_a \neq 0\} = ns+1$.
\end{theorem}
The last condition is satisfied only if $s=1$, $m=n$ and $P=\bbX_n^1$,
and then $P$ is a $(0,m,s)$-net.
\begin{proof}
For any function
$f: \bbK_n^s \to \mathbb{C}[z]$,
its Fourier transform
$\hat{f}: \bbX_n^s \to \mathbb{C}[z]$
is defined by
$$
\hat{f}(X):=\sum_{K \in \bbK_n^s} f(K)(K\bullet X),
$$
where $K\bullet X \in T \subset \bbC$.
Orthogonality of the characters stated in Lemma~\ref{lem:character}
implies that for a subgroup $P\subset \bbX_n^s$
$$
\frac{1}{\#(P)}
\sum_{X \in P}(K\bullet X)
=1 \mbox{ if $K \in P^\perp$ and }
=0 \mbox{ if $K \notin P^\perp$. }
$$
Thus, we have the so-called Poisson summation formula
$$
\frac{1}{\#(P)}
\sum_{X \in P}\hat{f}(X)
=\sum_{K \in \bbK_n^s}
\frac{1}{\#(P)}
\sum_{X \in P}
f(K)(K \bullet X)
=
\sum_{K\in P^\perp}
f(K).
$$
Now we put $f(K):=z^{\mu(K)}\in \bbC[z]$.
Then the right most end
is the weight enumerator polynomial
of $P^\perp$.
On the other hand, for $X=(\xi_{i,j})_{i,j} \in \bbX_n^s$,
$$
\hat{f}(X)=
\sum_{K \in \bbK_n^s} z^{\mu(K)} (K \bullet X)
=\prod_{i=1}^s \left(\sum_{k=(\kappa_1,\ldots,\kappa_n)\in \bbK_n}
z^{\mu((\kappa_1,\ldots,\kappa_n))}\prod_{j=1}^n \kappa_j(\xi_{i,j})\right).
$$
\mcomment{$\to$}
We compute the $i$-th component; namely,
we put $\xi_j:=\xi_{i,j}$,
and we prove in the following that the polynomial (which appeared in the above product)
$$
q((\xi_1,\ldots,\xi_n); z)
:=
\sum_{(\kappa_1,\ldots,\kappa_n)\in \bbK_n}
z^{\mu((\kappa_1,\ldots,\kappa_n))}\prod_{j=1}^n \kappa_j(\xi_j)
$$
is equal to $p(h;z)$, where $h = \mu^*(x_i)$. The coefficient of $z^a$ of $q((\xi_1,\ldots, \xi_n); z)$
is
$$
\sum_{\satop{k=(\kappa_1,\ldots,\kappa_n)\in \bbK_n}{\mu(k)=a}}
\prod_{j=1}^n\kappa_j(\xi_j)
=
\left(\sum_{\kappa \in G^*-\{0\}} \kappa(\xi_a) \right)
\prod_{j=1}^{a-1}{\sum_{\kappa \in G^*}\kappa(\xi_j)}.
$$
This is zero if there is $\xi_j \neq 0$ with $1\leq j \leq a-1$,
$(b^a-b^{a-1})$ if $\xi_1=\cdots=\xi_a=0$,
and
$-b^{a-1}$ if $\xi_1=\cdots=\xi_{a-1}=0, \xi_a \neq 0$,
because
$
 \sum_{\kappa \in G^*-\{0\}} \kappa(\xi_a)=(\sum_{\kappa \in G^*}
 \kappa(\xi_a)) -1 = b-1 \mbox{ or } -1
$, according to $\xi_a=0$ or not.
Thus, $q((\xi_1,\ldots,\xi_n); z)$ depends only
on the value $\mu^*(\xi_1, \ldots, \xi_n)$. Namely,
for $x=(\xi_1, \ldots, \xi_n)$ we have
$q(x;z)=p(\mu^*(x);z)$
where
$$
p(h;z)
:=
\left\{\begin{array}{rl} 1 + \sum_{a=1}^n z^a (b^a-b^{a-1}) & \mbox{if }
 h = 0,
 \\ 1 + \sum_{a=1}^{h-1} z^a (b^a-b^{a-1}) - z^hb^{h-1}
 & \mbox{if } h > 0. \end{array} \right.
$$
Now for $X=(x_1,\ldots,x_s)^\top \in \bbX_n^s$,
$$\hat{f}(X)=\prod_{i=1}^s p(\mu^*(x_i);z)$$
and hence
$$
\sum_{a=0}^{ns} N_a z^{a}
=
\frac{1}{\#(P)}
\sum_{X=(x_1,\ldots,x_s)^\top \in P}\prod_{i=1}^s p(\mu^*(x_i);z).
$$
\end{proof}

In the above proof, we showed: \jcomment{$\leftarrow$ replaced 'proved' with 'showed'}
\begin{lemma}
The Fourier \jcomment{$\leftarrow$ replaced 'fourier' with 'Fourier'} transform of $\bbK_n \to \bbC[z], k \mapsto z^{\mu(k)}$
is
$$
\widehat{z^{\mu(-)}}(x)=p(\mu^*(x),z).
$$
\end{lemma}
Indeed, we have a generalized version of Theorem~\ref{th:MacWilliams},
as follows.
Define a function $F$ from $\bbK_n^s$
to the polynomial ring of $sn$ variables:
$$
F:\bbK_n^s \to \bbC[z_{i,j} | 1\leq i \leq s, 1\leq j \leq n],
(k_1,\ldots,k_s) \mapsto \prod_{i=1}^s z_{i,\mu(k_i)}.
$$
A {\em generalized weight enumerator polynomial} for $P^\perp$
is defined by
$$
GW_{P^\perp}(z_{i,j}):=\sum_{K \in P^\perp} F(K).
$$

The Poisson summation formula then yields \jcomment{$\leftarrow$ replaced 'tells' with 'then yields'}
$$
GW_{P^\perp}(z_{i,j})=\frac{1}{\#(P)} \sum_{X \in P} \hat{F}(X).
$$
Here, by definition,
$$
\hat{F}(X)
=\sum_{K=(k_1,\ldots,k_s)^\top\in \bbK_n^s}
z_{1,\mu(k_1)}\cdots z_{s, \mu(k_s)}(K\bullet X).
$$
This equals
$$
\prod_{i=1}^s
(\sum_{k_i \in \bbK_n}
z_{i,\mu(k_i)}(k_i \bullet x_i)).
$$

Now, we define an $(n+1)$-variable polynomial
$\mathbf{p}(h; z_0,\ldots,z_n)$
by the summation  \jcomment{in the equation below: left hand side does not have $x$, so right hand side cannot have $x$ as well}
$$
\mathbf{p}(h; z_0,\ldots,z_n):= \sum_{a=0}^n z_a T_a(h),
$$
where $T_a(h)$ are integers defined as the coefficients of $z^a$
in the polynomial $p(h; z)$ in Theorem~\ref{th:MacWilliams}.
Then we have the equality
$$
\sum_{k \in \bbK_n}
z_{\mu(k)}(k \bullet x)
=
\mathbf{p}(\mu^*(x); z_0,\ldots,z_n).
$$
(The same proof as for Theorem~\ref{th:MacWilliams} works for proving this equality.
In fact,
$
p(h;z)=\mathbf{p}(h;z^0,z^1,\ldots, z^n)
$
holds.)
Thus, we proved
\begin{theorem} With the notation above we have \jcomment{$\leftarrow$}
$$
GW_{P^\perp}[z_{i,j}]:=
\sum_{K=(k_1,\ldots,k_s)^\top \in P^\perp}
z_{1,\mu(k_1)}\cdots z_{s, \mu(k_s)}
=
\frac{1}{\#(P)}
\sum_{X\in P}
\prod_{i=1}^s
\mathbf{p}(\mu^*(x_i); z_{i,0},\ldots,z_{i,n}).
$$
\end{theorem}
If we specialize the variables by putting $z_{i,a}:=z^a$ in
$GW_{P^\perp}(z_{i,j})$,
then we have Theorem~\ref{th:MacWilliams}.
If we specialize $z_{i,a}:=z_a$, then we have
Trinker's version of the MacWilliams identity
\cite{T08}.

The following corollary shows that,
to compute $t$-values of all the projections of $P$,
it suffices to compute a specialization
$\overline{GW}_{P^\perp}(z_1,\ldots,z_s)$
of $GW_{P^\perp}(z_{i,j})$
obtained by the substitution $z_{i,j} \leftarrow z_i^j$. That is,
\begin{align}\label{eq_overlineGW}
\overline{GW}_{P^\perp}(z_1,\ldots, z_s) := &
\sum_{K=(k_1,\ldots,k_s)^\top \in P^\perp}
z_{1}^{\mu(k_1)}\cdots z_{s}^{\mu(k_s)} \\
= &
\frac{1}{\#(P)}
\sum_{X\in P}
\prod_{i=1}^s
\mathbf{p}(\mu^*(x_i); z_{i}^{0}, z_{i}^1, \ldots,z_{i}^{n}). \nonumber
\end{align}

\begin{corollary}
Let $P\subset \bbX_n^s$. Take $u \subseteq \{1,\ldots, s\}$, $u \neq \emptyset$.
We may consider the projection $P_u$ of $P$
to the coordinates in $u$, which is the
image of $P$ by the projection $\bbX_n^s \to \bbX_n^u$
(here $\bbX_n^u$ is the set of mappings from $u$ to $\bbX_n$).
Then, the weight enumerator polynomial of
$P_u$ is obtained from $GW_{P^\perp}[z_{i,j}]$
by substituting $z_{i,j} \leftarrow z^j$ for $i \in u$,
$z_{i,0} \leftarrow 1$ for $i \notin u$,
and $z_{i,j} \leftarrow 0$ for $i \notin u$ and $1\leq j\leq n$.
Or equivalently, by substituting
$z_i \leftarrow z$ for $i \in u$
and
$z_i \leftarrow 0$ for $i \notin u$
in $\overline{GW}_{P^\perp}(z_1,\ldots,z_s)$ defined in \eqref{eq_overlineGW}.
\end{corollary}

\begin{proof}
We have a commutative diagram
of two short exact sequences of abelian groups:
$$
\begin{array}{ccccccccc}
0 & \to & P^\perp
  & \to & \bbK_n^s
  & \to & P^*
  & \to & 0 \\
  &     & \uparrow
  &     & \uparrow
  &     & \uparrow
  &     & \\
0 & \to & P_u^\perp
  & \to & \bbK_n^u
  & \to & P_u^*
  & \to & 0. \\
\end{array}
$$
The middle vertical arrow
$\bbK_n^u \to \bbK_n^s$
is obtained by supplementing $0$
on every $i$-th row for $i\notin u$.
The middle and the right vertical arrows are injective.
Hence, by diagram chasing, we have
$P_u^\perp = P^\perp \cap \bbK_n^u$.
By definition of $GW_{P_u^\perp}$
we have
$$
GW_{P_u^\perp}(z_{i,j})=
\sum_{K_u \in P_u^\perp}\prod_{i\in u}z_{i,\mu(k_i)}.
$$
By $P_u^\perp = P^\perp \cap \bbK_n^u$,
the above summation is over $K=(k_1,\ldots,k_s)^\top \in P^\perp$
satisfying $k_i=0$ for every $i\notin u$.
Compared with the definition  $GW_{P^\perp}[z_{i,j}]:=
\sum_{K=(k_1,\ldots,k_s)^\top \in P^\perp}
z_{1,\mu(k_1)}\cdots z_{s, \mu(k_s)}$,
we notice that $GW_{P^\perp}[z_{i,j}]$ with substitutions
\begin{equation*}
z_{i,j} \leftarrow \left\{\begin{array}{rl} 1 & \mbox{for } i \notin u, j = 0, \\ 0 & \mbox{for } i \notin u, 1 \le j \le n, \end{array} \right.
\end{equation*}
gives the $GW_{P_u^\perp}(z_{i,j})$.
Then, the substitution $z_{i,j} \leftarrow z^j$ for $i \notin u$
gives the weight enumerator polynomial of
$P_u^\perp$.
\end{proof}

From the polynomial $\overline{GW}_{P^\perp}(z_1,\ldots,z_s)$,
one can read which projection has large $t$-values.
More precisely, given an $s'$ with $1\leq s' \leq s$, we can compute
the (worst) choice of the $s'$ coordinates to which the projection
of $P$ have the largest exact $t$-value.

For a monomial $m$ in $\bbC[z_1,\ldots,z_s]$, we
define its {\em support} $\subset\{1,2,\ldots,s\}$
as the set of indices
of $z_i$ appearing \jcomment{$\leftarrow$ 'appearing' instead of 'appeared'} in $m$ (without counting the multiplicity).

\begin{corollary}
Let $\overline{GW}_{P^\perp}(z_1,\ldots,z_s)$ be as above.
Let $H_d$ be the (homogeneous) degree $d$ part
with respect to the total degree.
Suppose that $H_d\neq 0$. Take any monomial $m_d$ in $H_d$
such that the cardinality $c_d$ of its support is minimum
among the monomials in $H_d$. For $H_d=0$, we define $c_d=s+1$
(indeed, we may define $c_d$ as any number exceeding $s$).

For a given $s'$ with $1\leq s' \leq s$, define $d'$
as the minimum $d \geq 1$ satisfying  $c_d\leq s'$.
Then, the projection to the support of $m_{d'}$
has the exact $t$-value
$t':=m+1-d'$. This is the largest exact $t$-value
among all the projections $P_u$ with $\#(u)\leq s'$,
$\emptyset \neq u\subset \{1,\ldots,s\}$.
\end{corollary}

\begin{proof}
For a given nonempty $u\subset \{1,\ldots,s\}$,
we consider the projection $P_u$.
Suppose that $P_u^\perp \neq \{0\}$.
Let $a_0$ be the minimum degree of
nonconstant monomials
in the weight enumerator polynomial of $P_u^\perp$.
Then, the exact $t$-value of $P_u$ is
$m+1-a_0$ by Theorem~\ref{th:dual-net-minwt}.

The previous corollary implies that
the degree $d$ part of the weight enumerator polynomial of
$P_u^\perp$ is obtained from $H_d$
by substituting $z_i \leftarrow z$ for $i\in u$
and $z_i \leftarrow 0$ for $i \notin u$.
Thus, $a_0$ is the minimum positive integer $d$
such that $H_d$ does not vanish by this substitution.
Because the coefficients of $H_d$ are nonnegative
(being a part of weight enumerator polynomial),
the non-vanishing property is unchanged if we substitute
$z_i \leftarrow z_i$ for $i\in u$ and
$z_i \leftarrow 0$ for $i\notin u$.
For a given $H_d$, the $u$ with the minimal cardinality
satisfying this non-vanishing property is
given as the support of $m_d$; because if $m$ is a monomial appearing in
$H_d$, then taking $u$ as the support of $m$
we have the non-vanishing property. Conversely,
if we have the non-vanishing property, then there
is a monomial $m$ in $H_d$ whose support is contained in $u$.
Consequently, for a given $s'$ and for all $u$ with
cardinality $\leq s'$, the minimum $d$ such that
$H_d$ has the non-vanishing property with respect to $u$
is given as the minimum $d$ such that $c_d\leq s'$,
and by choosing $u$ to be the support of $m_d$.
\end{proof}

\subsection{On computing the $t$-value for general nets}

In Theorem~\ref{th:MacWilliams} we showed a MacWilliams type identity for digital nets over finite abelian groups. We now investigate this result when $P$ is not a subgroup of $\bbX_n^s$, but an arbitrary subset of $\bbX_n^s$ of size $b^m$.

For point sets $\mathcal{P} = \{\boldsymbol{x}_0,\ldots, \boldsymbol{x}_{b^m-1}\} \subset [0,1]^s$ it is known from Hellekalek~\cite{H94} that $\mathcal{P}$ is a strict $(t,m,s)$-net in base $b$ where
\begin{align*}
t = m  - \max\Big\{ & z\ge 0: \forall \boldsymbol{k} \in \{0,\ldots, b^{m}-1\}^s \mbox{ with } \mu(\boldsymbol{k}) \le z: \\ & \frac{1}{b^m} \sum_{l=0}^{b^m-1} {}_G\mathrm{wal}_{\boldsymbol{k}}(\boldsymbol{x}_l) = \int_{[0,1]^s} {}_G\mathrm{wal}_{\boldsymbol{k}}(\boldsymbol{x}_l) \, \mathrm{d} \boldsymbol{x} \Big\}.
\end{align*}
This means that strength of the net $m-t$ equals the Walsh degree of
exactness of the QMC rule based on the point set $\mathcal{P}$.
By the explanation after Definition~\ref{def:combinatorial-net},
we may fix any $n\geq m$
and let $P\subset \bbX_n^s$ be the approximation of $\mathcal{P}$
which is a combinatorial $(t,m,s)$-net.
Then, the above formula is equivalent to
\begin{align}
t = m  - \max\Big\{ & z\ge 0: \forall K \in \bbK_n^s
\mbox{ with } \mu(K) \le z: \nonumber \\
 &
 \frac{1}{\#(P)}
 \sum_{X \in P} K \bullet X
 =
 \frac{1}{\#(\bbX_n^s)} \sum_{X \in \bbX_n^s} K\bullet X
 \Big\}.
\label{eq:Helek}
\end{align}
\mcomment{???Do you know how to have only one number???} \jcomment{J: fixed}
Note that the last term is $1$ for $K=0$ and $0$ for $K \neq 0$.

\mcomment{In the following, I changed $\bbN_0$ into $\bbK$ etc.}
%

For a general point set $P$ one can still compute the polynomial
\begin{equation*}
\frac{1}{\#(P)}
\sum_{X=(x_1,\ldots,x_s)^\top \in P} \prod_{i=1}^s p(\mu^*(x_i);z)
=  \sum_{a=0}^{ms} \widehat{N}_a z^a,
\end{equation*}
where
\begin{equation*}
\widehat{N}_a = \sum_{K \in \bbK_n^s, \mu(K)=a}
\frac{1}{\#(P)} \sum_{X \in P}
K \bullet X.
\end{equation*}
\mcomment{??? I don't see an easy reason why this equality holds; I need:}\jcomment{changed the equations below}
To see this note that
\begin{align*}
\frac{1}{\#(P)}
\sum_{X=(x_1,\ldots,x_s)^\top \in P} \prod_{i=1}^s p(\mu^*(x_i);z)
= &
\frac{1}{\#(P)}
\sum_{X=(x_1,\ldots,x_s)^\top \in P} \widehat{z^{\mu(-)}}(X) \\
= &
\frac{1}{\#(P)}
\sum_{X=(x_1,\ldots,x_s)^\top \in P}
\sum_{K \in \bbK_n^s} z^{\mu(K)}K\bullet X \\ = & \sum_{a=0}^{ms} z^a  \sum_{K \in \bbK_n^s, \mu(K) = a} \frac{1}{\#(P)}
\sum_{X=(x_1,\ldots,x_s)^\top \in P}
  K\bullet X.
\end{align*}
For digital nets we have $N_a = \widehat{N}_a$ since the sum
$\frac{1}{\#(P)} \sum_{X \in P}K \bullet X$ takes on only the
values $1$ if $K \in P^\perp$ or $0$ otherwise. For general point sets this does not hold anymore.

On the other hand, for general $(t,m,s)$-nets,
$\widehat{N}_{a^\ast} \neq 0$ for some $a^\ast \ge 1$
implies that there is a $K \in \bbK_n^s$
with $\mu(K) = a^\ast$ such that
$\frac{1}{\#(P)} \sum_{X \in P}
K \bullet X \neq 0$.
Thus it follows from (\ref{eq:Helek}) that in this case
\begin{equation}\label{eq_gen_points}
t \ge m +1  - \min \{a | 1 \le a \le ns, \widehat{N}_a \neq 0\},
\end{equation}
where the minimum is defined as $ns+1$ if
$\widehat{N}_1 = \cdots = \widehat{N}_{ns} = 0$.
\mcomment{This may occur only when ???} This may only occur if $s=1$, $m = n$ and $P = \bbX_n^1$, in which case $P$ is a digital $(0,m,1)$-net.

Thus the method from the previous section can still be used to obtain a lower bound on the $t$-value for $(t,m,s)$-nets in base $b$.

\mcomment{???I need to take care of the below, it is straight forward???}
Note that for general point sets  equality in \eqref{eq_gen_points} may not hold. To see this, note that
\begin{equation*}
\frac{1}{b^m} \sum_{l=0}^{b^m-1} \prod_{i=1}^s p(\mu^*(x_{l,i});z),
\end{equation*}
only depends on $\boldsymbol{x}_l$ through the values $\mu^*(x_{l,i})$. Thus we can construct an example where the inequality is strict in the following way. Take a strict digital $(t,m,s)$-net in base $b$ with $t \le m-s-1$ \cite{DP10, niexi, SchSch}. Then shift all points in the elementary interval $\prod_{i=1}^s [1-b^{-1}, 1)$ to the point $(1-b^{-1},\ldots, 1-b^{-1})$, leaving the remaining points unchanged. This new point set has a $t$-value $t \ge m-s$, since the interval $[1-b^{-2}, 1) \times \prod_{i=2}^s [1-b^{-1}, 1)$ is empty. But the polynomial
\begin{equation*}
\frac{1}{b^m} \sum_{l=0}^{b^m-1} \prod_{i=1}^s p(\mu^*(x_{l,i});z)
\end{equation*}
stays unchanged, since the values of $\mu^*(x_{l,i})$ are the same for both point sets. Thus, the quality parameter $t'$ for the new point set satisfies
\begin{equation*}
t' \ge m-s > m-s-1 \ge t = m  - \min \{a | 1 \le a \le m, \widehat{N}_a \neq 0\}.
\end{equation*}

\subsection{The first algorithm for computing the $t$-value and weight  enumerator polynomial}\label{subsec_comp}

Theorem~\ref{th:MacWilliams} yields the following algorithm for
computing the exact $t$-value of a digital net (or finding a lower bound
on $t$ for general point sets). The algorithm works for digital nets
over finite abelian groups, and consequently for finite rings.
\mcomment{Any finite rings will do, no need to have Frobenius.}

\mcommentred{I am going to avoid
$\{\boldsymbol{x}_0,\ldots,
\boldsymbol{x}_{b^m-1}\}$ by using $P$ below.
This is direct except for 2 of Remark~3.
Even there, it is possible to replace with
$P_{m_1}\subset P_{m_2}\subset \cdots$.
Makoto would like to have Josef's opinion:
is it worth leaving $\{\boldsymbol{x}_0,\ldots,
\boldsymbol{x}_{b^m-1}\}$? At first I thought
that it might be helpful for the users
who wants to know only the algorithm.
If you think so, too, then I will let it as it is,
and add some explanation for interpreting $P$ to
a finite SEQUENCE (not a point set).
}
\mcomment{There is a confusion between $m$ and $n$,
but Makoto will fix them. (And $n$ is used for indexing
the sequence, too.)}

\jcomment{J: moved the following lemma forward to include this result in the algorithm}

\mcommentred{I noticed that following Lemma reduces the
order of computation significantly. It seems
comparable to your Algorithm~2. Thus, Makoto
would like to ask Josef whether Algorithm~2
is superior to the lemma below.
I do not yet follow your estimation
of computational complexity, in particular when we use
precomputation or lazy evaluation. The lemma below
should be put in ``algorithm'' section.
However, if Josef has time, it would be absolutely
better to check Lemma, hopefully by writing a code
and computer experimentation. In particular,
the speed comparison with Algorithm~2 would be
necessary.}

A straightforward computation shows the following:
\begin{lemma}\label{lem5} For $h \ge 0$ let $p(h; z)$ be the polynomials given in Theorem~\ref{th:MacWilliams}. \jcomment{$\leftarrow$ added} For $0<h\leq n$, we have
$$
p(h;z)=(1-z)\frac{1-(bz)^h}{1-(bz)}
$$
and
$$
p(0;z)=(1-z)\frac{1-(bz)^{m+1}}{1-(bz)} + b^m z^{m+1}.
$$
\end{lemma}
\jcomment{Corrected last formula for $p(0;z)$}

\jcomment{The text below is new and Algorithm 1 changed.}

Note that if one wants to compute the $t$-value, then it is sufficient to compute $N_a$ for $1 \le a \le m$, thus we can replace $p(0;z) = (1-z) \frac{1-(bz)^{m+1}}{1-(bz)} + b^mz^{m+1}$ by $p(0;z) = (1-z) \frac{1-(bz)^{m+1}}{1-(bz)}$. Then to compute the $t$-value, the computation of
\begin{equation}\label{eq_weight_enumerator3}
\frac{1}{\#P} \sum_{X \in P} \prod_{i=1}^s p(\mu^*(x_i); z)
\end{equation}
can be done using the following formula
\begin{equation*}
\left(\frac{1-z}{1-(bz)}\right)^s \frac{1}{\#P} \sum_{X \in P}\prod_{i=1}^s (1-(bz)^{\nu^\ast(x_{i})}),
\end{equation*}
where
\begin{equation*}
\nu^\ast(x_{i}) = \left\{\begin{array}{rl} \mu^\ast(x_{i}) & \mbox{if } \mu^\ast(x_{i}) > 0, \\ m+1 & \mbox{if } \mu^\ast(x_{i}) = 0. \end{array} \right.
\end{equation*}

We have the formal expansions
\begin{equation*}
\frac{1-z}{1-bz} = 1 + (b-1) z + (b^2-b) z^2 + \cdots + (b^m-b^{m-1}) z^m + \cdots
\end{equation*}
and
\begin{equation*}
\left(\frac{1-z}{1-bz}\right)^s = \sum_{a=0}^\infty z^a \sum_{\satop{c_1,\ldots, c_s \ge 0}{c_1 + \cdots + c_s = a}} \prod_{i=1}^s (b^{c_i} - \lfloor b^{c_i-1} \rfloor).
\end{equation*}
Since it is sufficient to compute \eqref{eq_weight_enumerator3} only up to degree $m$ for computing the $t$-value, it suffices to use the polynomial
\begin{align*}
Q_m(z) := & \left[1 + (b-1) z + (b^2-b) z^2 + \cdots + (b^m -b^{m-1}) z^m \right]^s \pmod{z^{m+1}} \\ = & \sum_{a=0}^m z^a \sum_{\satop{c_1,\ldots, c_s \ge 0}{c_1 + \cdots + c_s = a}} \prod_{i=1}^s (b^{c_i} - \lfloor b^{c_i - 1} \rfloor).
\end{align*}

\begin{algorithm}\label{alg1}
\begin{enumerate}
\item Given: digital net $\mathcal{P} = \{\boldsymbol{x}_0,\ldots, \boldsymbol{x}_{b^m-1}\} \subset [0,1)^s$ (over a finite abelian group $G$, finite field or finite Frobenius ring with $b$ elements). Let $\boldsymbol{x}_l = (x_{l,1}, \ldots, x_{l,s})$.
\item Compute the coefficients $N_a$ of $z^a$ for $0 \le a \le m$ of the polynomial
\begin{equation}\label{eq_weight_enumerator}
Q_m(z) \frac{1}{b^m} \sum_{l=0}^{b^m-1} \prod_{i=1}^s \left(1-(bz)^{\nu^\ast(x_{l,i})} \right) \pmod{z^{m+1}},
\end{equation}
where
\begin{equation*}
\nu^*(x) = \left\{\begin{array}{rl} \lceil -\log_b x \rceil & \mbox{if } x > 0, \\ m+1 & \mbox{if } x = 0. \end{array} \right.
\end{equation*}
\item Let the coefficient of $z^a$ of the weight enumerator polynomial \eqref{eq_weight_enumerator} be $N_a$. Then compute
\begin{equation*}
t = m- \min \{a | 1 \le a \le m, N_a \neq 0\},
\end{equation*}
where the minimum is defined to be $m+1$ if $N_1= \cdots = N_m = 0$.
\item Return $t$.
\end{enumerate}
\end{algorithm}

\begin{remark}
If one uses Algorithm~\ref{alg1} for a general point set $\mathcal{P}  = \{\boldsymbol{x}_0,\ldots, \boldsymbol{x}_{b^m-1}\} \subset [0,1)^s$, then the returned value $t$ is a lower bound on the quality parameter of the point set, i.e., it implies that $\mathcal{P}$ is not a $(t-1,m,s)$-net in base $b$.
\end{remark}

\begin{remark}\label{rem3}
\begin{enumerate}
\item Note that the proof of Theorem~\ref{th:MacWilliams} can be modified by setting
\begin{equation*}
\hat{f}(X) := \sum_{K \in \bbK_{\ell}^s} f(K) (K \bullet X),
\end{equation*}
where one can choose $\ell \ge m-t$ (for instance if $t_0$ is a known lower bound for the given net, then one can choose $\ell = m-t_0$). One then obtains the polynomials
\begin{align*}
p_{\ell}(1;z) & = 1-z, \\
p_{\ell}(2;z) & = 1 + (b-1)z - b z^2, \\
\ldots & \ldots, \\
p_{\ell}(\ell;z) & = 1 + (b-1)z + (b^2-b) z^2 + \cdots + (b^{\ell-1} - b^{\ell-2}) z^{\ell-1} - b^{\ell-1} z^\ell, \\
p_{\ell}(0;z) & = 1 + (b-1)z + (b^2-b) z^2 + \cdots + (b^{\ell-1}-b^{\ell-2}) z^{\ell-1} + (b^\ell-b^{\ell-1}) z^\ell.
\end{align*}
and one computes the polynomial
\begin{equation*}
\frac{1}{b^m} \sum_{l=0}^{b^m-1} \prod_{i=1}^s p_{\ell}(\mu^*(x_{l,i});z),
\end{equation*}
where $\mu(x_{l,i}) = \lceil - \log_b x \rceil$, \jcomment{$\leftarrow$ changed formula} and
\begin{equation*}
t = m + 1 - \min\{a|1 \le a \le \ell: N_a\neq 0\},
\end{equation*}
where the minimum is defined to be $\ell+1$ if $N_1 = \cdots = N_\ell = 0$ (since we assume that $\ell \ge m-t$, $N_1=\cdots = N_\ell=0$ can only happen if the exact $t$-value is $m-\ell$). Again, this computation can be simplified using Lemma~\ref{lem5}. \jcomment{$\leftarrow$ added sentence}

This way one can reduce the computational cost of calculating the weight enumerator polynomial modulo $z^{m+1}$.

\item On the other hand, if one wants to compute the $t$-values of a point set of $b^{m_u}$ points for several values of $m \in \{m_1,\ldots, m_u\}$ ($m_1 \le m_2 \le \cdots \le m_u$), then one can, for instance, choose $\ell = m_u$ for all cases. In this case, the sum $\sum_{l=0}^{b^{m_r}-1} \prod_{i=1}^s p_{\ell}(\mu^*(x_{l,i});z)$ can be reused when computing $\sum_{l=0}^{b^{m_{r+1}} -1} \prod_{i=1}^s p_{\ell}(\mu^*(x_{l,i});z)$, i.e., one only needs to compute $\sum_{l= b^{m_r}}^{b^{m_{r+1}}-1} \prod_{i=1}^s p_{\ell}(\mu^*(x_{l,i});z)$ and add it to the previous result for the sum $\sum_{l=0}^{b^{m_{r}}-1} \prod_{i=1}^s p_{\ell}(\mu^*(x_{l,i});z)$.
\end{enumerate}
\end{remark}

We now discuss the case when computing $N_a$ for $0 \le a \le D$ with $D > m$, i.e., the computation of the weight enumerator polynomial. Using Lemma~\ref{lem5}, we can improve the order of the computational complexity of computing the weight enumerator polynomial as follows. Consider the special case where
\begin{equation}\label{eq:simple-case}
\mbox{
Any row of $X \in P \setminus \{0\}$ is non zero. \jcomment{J: replaced $P^\perp$ by $P$ and $K$ by $X$}
}
\end{equation}
Then in the computation of \jcomment{added the following}
\begin{equation}\label{eq_weight_enumerator2}
\frac{1}{\#P} \sum_{X \in P \setminus \{0\} } \prod_{i=1}^s p(\mu^*(x_i); z)
\end{equation}
no $p(0;z)$ is involved. Thus, we may compute (\ref{eq_weight_enumerator2}) as follows. Let $Z$ be a new variable, which will be substituted by $bz$ later. Compute \jcomment{changed lower summation index from $0$ to $1$}
$$
R(Z):=\frac{1}{(1-Z)^s}
\sum_{l=1}^{b^m-1} \prod_{i=1}^s (1-Z^{\mu^*(x_{l,i})}).
$$
Note that dividing a polynomial by $1-Z$ is an easy task.
Then 
$$
\mathrm{WP}_{P^\perp}(z)= (1-z)^s \frac{1}{b^m} R(bz) + \frac{1}{b^m} p(0;z)^s.
$$
\mcommentred{
Makoto is curious about whether
$\frac{1}{b^m} R(bz)$ has integer coefficients or not.
If so, $R(Z)$ has only terms of the form $z^{mk}$, $k=0,1,\ldots$,
which is strange. Thus, Makoto guesses that
it becomes integer, after multiplying $(1-z)^s$. } \jcomment{J: The equation changed, so one needs to multiply with $(1-z)^s$ first and then add $p(0;z)^s$ to get $W_{P^\perp}$.}

\jcomment{There are two improvements possible: 1. We can avoid high powers of $b$,
to prevent a ``coefficient explosion'' slowing down the algorithm.
2. Multiplication is very simple. } \jcomment{$\leftarrow$ Should this be a remark?}

Now we consider the general case.
In computing $\prod_{i=1}^s p(\mu^*(x_{l,i});z)$
in
\begin{equation*}
\frac{1}{b^m} \sum_{l=0 }^{b^m-1} \prod_{i=1}^s p(\mu^*(x_{l,i}); z)
\end{equation*}
we count
the number $\ell(\boldsymbol{x}_l)$ of $i$ with $\mu^*(x_{l,i}) = 0$. \jcomment{$\leftarrow$ J: don't really understand this sentence}
We prepare $s$ memory for the polynomials.
The $0$th one is accumulating the sum
of $\prod_{i=1}^s p(\mu^*(x_{l,i});z)$
for $\boldsymbol{x}_l$ with $\ell(\boldsymbol{x}_l)=0$, by the
method described as above.
The $r$th one is accumulating the
sum for $\boldsymbol{x}_l$ with $\ell(x_l)=r$.
The product is separated into two parts,
the product of those $x_{l,i}$ with
$\mu(x_{l,i})>0$ (for which we can use the same trick
based on new variable $Z$
as above) and the product of those $x_{l,i}$ with $\mu(x_{l,i})=0$.
We can factor out the latter (since the number is $r$),
and may sum the former terms (with fixed $r$).
After exhausting all $0 \le l < b^m$, finally, we can
substitute $Z:=bz$ for each of the $s$ polynomials, and add them up.

Note that the above approach for computing the $t$-value can also be used for the generalized MacWilliams.

\mcommentred{This is the end of my new comment IN THIS SECTION.}

We now investigate the complexity of computing the $t$-value via the weight enumerator polynomial $\frac{1}{b^m} \sum_{l=0}^{b^m-1} \prod_{i=1}^s p(\mu^*(x_{l,i});z)$. To do so, the coefficients of the product $\prod_{i=1}^s (1-(bz)^{\nu^\ast(x_{l,i})})$ only need to be computed up to $z^{m}$ (where the coefficient is bounded by $b^{ms+m}$, which can be obtained by estimating $WP_{P^\perp}(1)$) and therefore can be computed in $\mathcal{O}(m)$ operations. Thus the $t$-value can be computed in $\mathcal{O}(N s \log N)$ operations.

The computational complexity for computing the minimum distance
in a binary linear code is $NP$-hard \cite{Va97} in terms of the
dependence on the dimension. This implies that computing the $t$-value
of a digital net in base $2$ is also $NP$-hard.
\mcomment{If you assume $NP\neq P$ and even a stronger conjecture:
$NP$ does not necessarily mean exponentially difficult; there
are notions of sub-exponential.}
\footnote{Notice that the degree of the polynomials is in  practice up to, say $30$ (which yields $2^{30} \approx 10^9$ points). Even if one requires more detailed knowledge of the weights, the degree of the polynomials is likely below several hundreds. Thus polynomial multiplication algorithms by Sch\"onhage and Strassen~\cite{SS71}, Cantor and Kaltofen~\cite{CK91} or the Toom-Cook algorithm are probably not beneficial for these computations, although asymptotically they have a better performance.} 

\subsection{Using an inverse MacWilliams identity to compute the $t$-value}


Although we stated the MacWilliams identity for digital nets and the dual group, it can be understood as a relationship between a group and its dual group. As such, it is quite obvious that a MacWilliams identity as stated in Theorem~\ref{th:MacWilliams} in the reverse direction is possible. However, the relationship of the weight enumerator polynomial with the $t$-value is lost. In the following we prove that for a certain choice of weight on the other hand, there does exist a relationship between the sum of products of polynomials ($\frac{1}{b^m} \sum_{n=0}^{b^m-1} \prod_{i=1}^s p(\mu^*(x_{n,i});z)$ in Theorem~\ref{th:MacWilliams}) and the $t$-value for the reverse statement of the MacWilliams identity. This is shown in the following theorem.

\begin{theorem}\label{thm_inv_weight}
Let $G$ be a finite abelian group consisting of $b$ elements. Let $P \subseteq \bbX_n^s$ be a sub group of size $b^m$. For $X \in \bbX_n^s$, let $x_i \in \bbX_n$ be its $i$-th row, namely, $X = (x_1,\ldots, x_s)^\top \in \bbX_n^s$. Then $\mathcal{P}$ is a strict digital $(t,m,s)$-net over $G$ with
\begin{align*}
t = & (1-s)(m+1) + \deg\Bigg(- \left(1 + (b-1) z + \cdots + (b^m-b^{m-1})z^m - b^{m} z^{m+1}\right)^s \\ & + b^{sm-m} \sum_{X \in P} z^{(m+1)s} \prod_{i=1}^s \left(z^{-\nu^\ast(x_i)}  - 1 \right)   \Bigg),
\end{align*}
where $\nu^*(x_i) = \mu^*(x_i)$ for $x_i \neq 0$ and $\nu^*(0) = m+1$.
\end{theorem}

\begin{proof}
For $k \in \bbK_m$ let
\begin{equation*}
q(k;z) = \sum_{x \in \bbX_m} z^{m+1-\nu^*(x)} (k \bullet x) - z^{m+1} \sum_{x \in \bbX_m} k \bullet x.
\end{equation*}

If $k = 0$ we have
\begin{align*}
p(0;z) := q(0; z) = & \sum_{x \in \bbX_m} z^{m+1-\nu^*(x)} - z^{m+1} \sum_{x \in \bbX_m} 1 =   1 + \sum_{a=1}^m (b^a-b^{a-1}) z^a - b^m z^{m+1}.
\end{align*}

Assume now that $k \neq 0$. Then
\begin{equation*}
q(k;z) = 1 + \sum_{a=1}^{m} z^{a} \sum_{\satop{x \in \bbX_m}{a = m+1-\nu^*(x)}} k \bullet x - z^{m+1} \sum_{x \in \bbX_m} k \bullet x.
\end{equation*}
Since $k \neq 0$ we have $\sum_{x \in \bbX_m} k \bullet x = 0$. We consider now the double sum. Let $k = (\kappa_1,\ldots, \kappa_m)$ and $x = (\xi_1,\ldots, \xi_m)$. The condition $a = m+1-\nu^*(x)$ implies that $x = (0,\ldots, 0, \xi_{m+1-a}, \ldots, \xi_m)$ with $\xi_{m+1-a}  \neq 0$. Then
\begin{align*}
\sum_{\satop{x \in \bbX_m}{a = m+1-\nu^*(x)}} k \bullet x & = \left( \sum_{\xi_{m+1-a} \in G-\{0\}} \kappa_{m+1-a}(\xi_{m+1-a}) \right) \prod_{i=1}^{a-1} \sum_{\xi_{m+1-i} \in G} \kappa_{m+1-i}(\xi_{m+1-i}) \\ = & \left\{\begin{array}{rl} 0 & \mbox{if } \kappa_{i} \neq 0 \mbox{ for an } i \in \{m+2-a,\ldots, m\}, \\ b^a-b^{a-1}  & \mbox{if } \kappa_{m} = \cdots = \kappa_{m+1-a} = 0, \\ - b^{a-1} & \mbox{if } \kappa_{m} = \cdots = \kappa_{m+2-a} = 0, \kappa_{m+1-a} \neq 0. \end{array} \right.
\end{align*}
Thus $q(k;z)$ depends only on $\mu(k)$ and we have
\begin{equation*}
p(\mu(k);z) := q(k;z) = 1 +  \sum_{a=1}^{m-\mu(k)} (b^a-b^{a-1}) z^{a} - b^{m- \mu(k)} z^{m+1-\mu(k)}.
\end{equation*}
Thus, for $0 \le h \le m$ we have
\begin{equation*}
\deg(p(h;z)) = m+1-h
\end{equation*}
and the coefficient of $z^{m+1-h}$ of $p(h;z)$ is negative in all cases.

Let
\begin{equation*}
Q(z) :=  \sum_{K \in P^\perp - \{0\} } \prod_{i=1}^s p(\mu(k_i);z).
\end{equation*}
Then, since the leading coefficients of the polynomials $p(\mu(k_i);z)$ are always negative, it follows that
\begin{equation*}
\deg(Q) =  \max_{K \in P^\perp - \{0\} } \deg \left(\prod_{i=1}^s p(\mu(k_i);z)\right) = s(m+1) - \min_{K \in P^\perp - \{0\} } \mu(K),
\end{equation*}
which is the same as
\begin{equation*}
\min_{K \in P^\perp - \{0\} } \mu(K) = s(m+1) - \deg(Q).
\end{equation*}
Thus Theorem~\ref{th:dual-net-minwt} implies that
\begin{equation*}
t = m+1 - \min_{K \in P^\perp - \{0\} } \mu(K) = (1-s)(m+1) + \deg(Q).
\end{equation*}

We now find a quick way of computing the polynomial $Q$. We have
\begin{align*}
\frac{1}{\#(P^\perp)} (Q(z) + p(0;z)^s) = & \frac{1}{\#(P^\perp)}  \sum_{K \in P^\perp} \prod_{i=1}^s \left(\sum_{x \in \bbX_m} (z^{m+1-\nu^*(x)}-z^{m+1}) (k_i \bullet x) \right) \\  = & \sum_{X \in \bbX_m^s} z^{(m+1)s} \prod_{i=1}^s (z^{-\nu^*(x_i)}-1) \frac{1}{\#(P^\perp)} \sum_{K \in P^\perp} K\bullet X \\ = & \sum_{X \in P} z^{(m+1)s} \prod_{i=1}^s (z^{-\nu^*(x_i)} - 1).
\end{align*}
Thus
\begin{equation*}
Q(z) = -(p(0;z))^s + \#(P^\perp) \sum_{X \in P} z^{(m+1)s} \prod_{i=1}^s (z^{-\nu^*(x_i)}-1).
\end{equation*}
\end{proof}

\subsection{The second algorithm for computing the $t$-value}

We now present an algorithm for computing the $t$-value of a digital net based on Theorem~\ref{thm_inv_weight}. Again the algorithm works for digital nets
over finite abelian groups, and consequently for finite rings. 

\begin{algorithm}\label{alg2}
\begin{enumerate}
\item Given: digital net $\mathcal{P} = \{\boldsymbol{x}_0,\ldots, \boldsymbol{x}_{b^m-1}\} \subset [0,1)^s$ (over a finite abelian group $G$, finite field or finite Frobenius ring with $b$ elements). Let $\boldsymbol{x}_l = (x_{l,1}, \ldots, x_{l,s})$.
\item Compute the coefficients of $z^a$ for $(s-1)(m+1) \le a \le s(m+1)-1$ of the polynomial
\begin{equation*}
Q(z) = -(1+(b-1)z + \cdots + (b^m-b^{m-1}) z^m - b^m z^{m+1})^s + b^{sm-m} \sum_{l=0}^{b^m-1} \prod_{i=1}^s (z^{\mu(x_{l,i}b^m)} - z^{m+1}),
\end{equation*}
where $\mu(0) = 0$ and for a positive integer $u \ge 1$ we have $\mu(u) = 1 + \lfloor \log_b u \rfloor$.
\item Compute $t = (1-s)(m+1) + \deg(Q)$.
\item Return $t$.
\end{enumerate}
\end{algorithm}


As opposed to the first algorithm, Algorithm~\ref{alg2} is not extensible as explained in Remark~\ref{rem3}, item 2). The value of $m$ can however be adjusted in a similar manner as in Remark~\ref{rem3} for the first algorithm. If doing so one needs to use $\mu(x_{l,j} b^{m'})$, where $m'$ is the new depth.

The computationally most expensive step in the algorithm is $2.$, where products of polynomials need to be computed for $0 \le l < b^m$. Computing the product in general requires $\mathcal{O}(s^2 m)$ operations. However, only the $m+1$ most significant coefficients need to be computed to obtain the $t$-value and these $m+1$ coefficients can be computed in $\mathcal{O}(m s)$ operations. This can, for instance, be done by computing the first $m+1$ coefficients of the product of reciprocals
\begin{equation*}
\prod_{i=1}^s (-1+y^{m+1-\mu(x_{l,i}b^m)}),
\end{equation*}
which are the same as the $m+1$ most significant coefficients of $\prod_{i=1}^s (z^{\mu(x_{l,i}b^m)}-z^{m+1})$, and then storing them in the correct place.
Thus the computational cost of the algorithm is $\mathcal{O}(N s \log  N)$ operations.

\jcomment{The following text is commented out, since it is also commented out for Algorithm 1.}

\jcomment{
As for Algorithm~\ref{alg1}, the values $\mu(x_{n,i}b^m)$ have a similar distribution as analyzed above. This information can be used to precompute some products $\prod_{i=1}^s (z^r-z^{m+1})$ which frequently occur, thus reducing the computational cost of the algorithm further.
}

In Section~\ref{sec:ring-dual} we will discuss the relation between our character-theoretic dual $P^\perp$ and the dual net defined by a ring-theoretic inner product \mcomment{(???symbol)} introduced in \cite{np}.
\mcomment{??? I need to check that
they dealt only finite fields.} \jcomment{Josef: I checked it, they only deal with finite fields}
It is proved that for a wide class of finite rings
these notions coincide in a suitable sense. However, there are counter
examples for general finite commutative rings $R$ and $P$ in which
the ring-theoretic dual is strictly larger than the character-theoretic dual.
On the other hand, it is proved that if $P$ is a free $R$-module, the
formula in Theorem~\ref{th:dual-net-minwt} holds if $P^\perp$ is
replaced by the ring-theoretic dual net.


\section{Numerical result}\label{sec:num-res}

As a proof of concept we computed the $t$-values of digital nets obtained from a Sobol' sequence as implemented in Matlab 2011a. The results are presented in the table below. In these experiments, Algorithm~\ref{alg2} is slightly faster than Algorithm~\ref{alg1} (although both algorithms have not been optimized in our experiments).

The table also contains the exact $t$-values of the Sobol' sequence as computed in \cite{DN} for dimension up to $10$ in the first column. We remark that using different direction numbers changes the $t$-values for particular $m$ and $s$, see \cite{DN}.

\begin{table}
\begin{tabular}{|r||r|r|r||r|r|r|r|r||r|r|r|r|r||r|r|r|r|r|r|r|}
  \hline
$m \backslash s$ & 3 & 4 & 5 & 6 & 7 & 8 & 9 & 10 & 11 & 12 & 13 & 14 & 15 & 16 & 17 & 18 & 19 & 20 & 21 & 22 \\ \hline \hline
2 & 1 & 1 & 1 & 1 & 1 &  1 &  1 & 1 & 1 & 1 & 1 & 1 & 1 & 1 & 1 & 1 & 1 & 1 & 1 & 1 \\ \hline
3 & 1 & 2 & 2 & 2 & 2 &  2 &  2 & 2 & 2 & 2 & 2 & 2 & 2 & 2 & 2 & 2 & 2 & 2 & 2 & 2 \\ \hline
4 & 1 & 1 & 3 & 3 & 3 &  3 &  3 & 3 & 3 & 3 & 3 & 3 & 3 & 3 & 3 & 3 & 3 & 3 & 3 & 3 \\ \hline
5 & 1 & 2 & 2 & 2 & 3 &  3 &  3 & 4 & 4 & 4 & 4 & 4 & 4 & 4 & 4 & 4 & 4 & 4 & 4 & 4 \\ \hline\hline
6 & 1 & 2 & 3 & 3 & 3 &  4 &  4 & 5 & 5 & 5 & 5 & 5 & 5 & 5 & 5 & 5 & 5 & 5 & 5 & 5 \\ \hline
7 & 1 & 2 & 3 & 4 & 4 &  4 &  4 & 6 & 6 & 6 & 6 & 6 & 6 & 6 & 6 & 6 & 6 & 6 & 6 & 6 \\ \hline
8 & 1 & 3 & 3 & 4 & 4 &  5 &  5 & 5 & 5 & 5 & 5 & 5 & 5 & 5 & 5 & 5 & 5 & 6 & 6 & 6 \\ \hline
9 & 1 & 3 & 3 & 4 & 5 &  5 &  5 & 5 & 6 & 6 & 6 & 6 & 6 & 6 & 6 & 6 & 6 & 7 & 7 & 7 \\ \hline
10& 1 & 3 & 4 & 4 & 4 &  6 &  6 & 6 & 7 & 7 & 7 & 7 & 7 & 7 & 7 & 7 & 7 & 7 & 7 & 8 \\ \hline\hline
11& 1 & 3 & 5 & 5 & 5 &  6 &  6 & 6 & 8 & 8 & 8 & 8 & 8 & 8 & 8 & 8 & 8 & 8 & 8 & 8 \\ \hline
12& 1 & 3 & 4 & 5 & 5 &  7 &  7 & 7 & 9 & 9 & 9 & 9 & 9 & 9 & 9 & 9 & 9 & 9 &9&9 \\ \hline
13& 1 & 3 & 5 & 6 & 6 &  7 &  7 & 7 & 9 & 9 & 9 & 9 & 10 & 9 & 9 & 9 & 9 & 10&10&10  \\ \hline
14& 1 & 3 & 4 & 6 & 7 &  7 &  8 & 8 & 9 & 9 & 9 & 9 & 9 & 10 & 10 & 10 & 10 & 10&10&10  \\ \hline
15& 1 & 3 & 5 & 5 & 6 &  8 &  9 & 9 & 9 & 9 & 9 & 9 & 9 & 9 & 10 & 10 & 10 & 10&11&11  \\ \hline\hline
16& 1 & 3 & 4 & 6 & 7 &  9 &  9 & 9 & 9 &10 &10 &10 &10 &9 &9 &10 & 10 & 11&12&12 \\ \hline
17& 1 & 3 & 5 & 7 & 8 &  8 &  8 &10 &10 & 11&11 &11 & 11 & 10 &10 &10 &11 &12&12&12 \\ \hline
18& 1 & 3 & 4 & 7 & 7 &  8 &  9 &10 &10 & 11&12 &12 & 12 & 11 & 11 &11 &11 &13&13&13\\ \hline
19& 1 & 3 & 5 & 7 & 7 &  8 & 10 &10 &10 & 11&12 &12 & 12 & 12 & 12 & 12 &13 &14&14&14 \\ \hline
20& 1 & 3 & 4 & 7 & 8 & 9  & 11 &11 &11 &11 &13 &13 & 13 & 12 & 12 & 12 & 13 &14&14&14 \\ \hline\hline
21& 1 & 3 & 5 & 6 & 7 &10  & 12 &12 & 12& 12&12 &12 &12 &12 &12 &13 & 14 &15&15&15 \\ \hline
22& 1 & 3 & 5 & 7 & 8 & 10 &11  &11 & 12& 12&13 &13 &13 &13 &13 &13 &13 & 16&16&16 \\ \hline
23& 1 & 3 & 5 & 7 & 8 & 11 & 11 &12 &13 & 13&14 &14 &14 &14 &14 &14 &14 &15&15&15  \\ \hline
24& 1 & 3 & 5 & 8 & 9 &11  &12  &12 & 12& 13&15 &15 &15 &15 &15 &15 &15 &16&16&16 \\ \hline
25& 1 & 3 & 5 & 7 & 9 & 10 &11  &13 & 13& 13&16 &16 &16 &16 &16 &16 &16 &17&17&17 \\ \hline \hline
$t_S$ & 1 & 3 & 5 & 8 & 11 & 15 &19 & 23&  &   &   & & & & & & & & & \\
\hline
\end{tabular}
\caption{The exact $t$-values for digital nets obtained from Sobol' sequence as implemented in Matlab 2011a. The last row (marked by $t_S$) contains the $t$-values of the Sobol' sequence from \cite{DN}.}
\end{table}


\section{Nets over finite abelian groups and over finite rings}\label{sec:group-ring}

In this section we generalize the definition of $(t,m,s)$-nets and study digital nets over finite abelian groups and show how this theory relates to digital nets defined over finite rings.

\subsection{A generalization to (T,M,s)-nets}
As considered in Section~\ref{sec:net-group},
our framework is a point set
$P \subset \bbX_n^s=M_{s,n}(G)$,
which is injectively mapped
by $\varphi^{-1}: \bbX_n^s \to [0,1)^s$
into the $s$-dimensional cube.

Note that the cardinality of a subgroup $P$ of $\bbX_n^s$ is not necessarily
of the form of $b^m$, and the cardinality of
the points in the elementary $b$-adic intervals
need not be of the form of $b^t$. This leads us to the following definition.

\begin{definition}\label{lem:net-by-uniformness-general}
Let $M$ be a positive integer and for any integers $d_1,\ldots,d_s\geq 0$,
let $Y:=Y(d_1,\ldots,d_s)$ denote the union of
$$\{(1,1), (1,2),\ldots,(1,d_1)\},
\{(2,1), (1,2),\ldots,(1,d_2)\},
\ldots,
\{(s,1), (1,2),\ldots,(1,d_s)\}.$$

A finite subset $P\subset \bbX_n^s$ of
cardinality $M$ is a (combinatorial)
$(T,M,s)$-net in base $b$ if and only if for any $b^{d_1+d_2+\cdots+d_s} \leq M/T$,
the composition $P \subset \bbX_n^s \to \bbX_n^s(Y)$ is uniform. If $T$ is the smallest value such that $P$ is a $(T,M,s)$-net in base $b$, then we call $P$ a strict $(T,M,s)$-net in base $b$.
\end{definition}

In this case,
for any integer $d\leq M/T$,
the number of points
in each elementary b-adic interval of volume
$b^{-d}$
has the same number (depending only on $d$,
namely $M/b^d$) of points in $\varphi^{-1}(P)$.

In this notation, the original $(t,m,s)$-net is a $(b^t,b^m,s)$-net. On the other hand, let $P$ be a $(t,m,s)$-net in base $b$ for some admissible parameters $t,m,s,b$. Let $\lambda \ge 1$ be an integer and let the multiset $P_\lambda$ be the point set where each point is taken with multiplicity $\lambda$. Then $P_\lambda$ is a $(\lambda b^t,\lambda b^m, s)$-net in base $b$. If $P$ is a strict $(t,m,s)$-net in base $b$, then $P_\lambda$ is a strict $(\lambda b^t, \lambda b^m, s)$-net in base $b$. For $\lambda$ not of the form $b^k$, this provides examples of point sets satisfying Definition~\ref{lem:net-by-uniformness-general}.

\subsection{Generating basis and heterogenous case}

Let $P$ be a subgroup of $\bbX_n^s$. In practice, we need to enumerate the points in $P$.
A possible way is to find a generating set of $P$
as an abelian group.
More precisely, the structure theorem
of a finite abelian group states that
$$
P \cong \prod_{i=1}^r \bbZ_{q_i},
$$
where $q_i$ is a prime power.
Then, we precompute the $r$ elements of $P$
corresponding to $(1,0,\ldots,0)$, $\ldots$,
$(0,0,\ldots,0,1)$. By a lexicographic
enumeration of the elements of $\prod_{i=1}^r \bbZ_{q_i}$
we can enumerate points in $P$ through the above
isomorphism.
In practice, it may be desirable that for any $1\leq r'\leq r$,
the image of $\prod_{i=1}^{r'} \bbZ_{q_i}$ has
a low discrepancy property. Then, we can apply QMC-integration
for the first $\prod_{i=1}^{r'} \bbZ_{q_i}$ points,
and if the result is not satisfactory, we may increase $r'$,
similarly as for $(t,s)$-sequences \cite[Chapter~4]{DP10}.

We remark that in the results so far, we may
choose different (i.e. non-isomorphic as
abelian groups) $G$ of the same cardinality $b$
in the coordinate of $\bbX_n=M_{s,n}(G)$.
Namely, we may choose $n\times s$ possibly different
finite abelian groups $G_{i,j}$.
By taking its dual $G_{i,j}^*$, all the results so far
hold.

We may even change the cardinality of $G_{i,j}$,
so that $b_{i,j}:=\#(G_{i,j})$ depends on $(i,j)$.
The embedding of the
$i$-th row $\prod_{j=1}^n G_{i,j}$
into the $i$-th coordinate in the interval $[0,1)$ is given
as follows: first equating $[0,1)$
into $b_{i,1}$ intervals, then each interval is
equated into $b_{i,2}$ intervals, and so on.
For a subgroup
$P \subset \prod_{1\leq i \leq s, 1\leq j \leq n} G_{i,j}$,
we can define
$P^\perp  \subset \prod_{1\leq i \leq s, 1\leq j \leq n} G_{i,j}^*$.
We can define elementary $(b_{i,j})$-adic intervals.
For a subset $Y \subset \{1,2,\ldots,s\}\times \{1,2,\ldots,n\}$,
we define its co-volume
$\mathrm{Vol}(Y):= \prod_{(i,j)\in Y}b_{i,j}$.

A point set $P$ is then a $(T,M,s)$-net
if the cardinality of $P$ is $M$ and
the mapping from $P$ to $\prod_{(i,j)\in Y} G_{i,j}$
is uniform for any $Y$ of the
form $Y(d_1,\ldots,d_s)$ whose co-volume is smaller than or
equal to $M/T$.

A version of an NRT-weight and a MacWilliams-type identity
can be defined and proved, but we omit its explicit description.

\subsection{Ring-theoretic dual net}\label{sec:ring-dual}
The original notion of the dual net \cite{np}
is defined in the case where the digital net is defined using a finite field,
and the definition uses an inner product.
We introduce a straight forward generalization
to a finite ring. We use the letter $R$ for finite rings in the following and consider digital nets over $R$ (rather than $G$). Hence in the following we use
\begin{align*}
\bbX_n & := \{x = (\xi_1,\ldots, \xi_n): \xi_i \in R\}.
\end{align*}

\begin{definition}
Let $R$ be a finite ring. Then,
$\bbX_n^s$ is a free left $R$-module of rank $sn$.
Let $P \subset \bbX_n^s$ be a sub $R$-module.
If $P$ is a $(t,m,s)$-net in the
sense of the combinatorial $(t,m,s)$-net
(Definition~\ref{def:tms-net}), then we call $P$ a digital $(t,m,s)$-net
over the ring $R$.
\end{definition}
This definition is slightly more general
than that given in \cite{LNS}, which
treats the case of a free $R$-module $P$.

Recall that $P \subset \bbX_n^s$ is
{\em free} of rank $m$ if there are $X_1,X_2, \ldots,X_m \in P$
such that every element $X\in P$ is
uniquely represented by an $(x_1,\ldots,x_m) \in R^m$
as
$$
X=x_1X_1+\cdots+x_mX_m.
$$
Namely, $X_1,\ldots,X_m$ are linearly independent
over $R$ and generate $P$ as an $R$-module.
\begin{definition}
Let $C_1, \ldots, C_s \in M_{n,m}(R)$ be matrices and let $X_1,\ldots,X_m$ be defined by:
the $j$-th row of $X_i$ ($n$-dimensional)
is the transpose of the
$i$-th column of $C_j$. Assume that $X_1, \ldots, X_m$ are a free basis. Then we call the set $P \subset \bbX_n^s$ which is generated by the
free basis $X_1,\ldots,X_m$,
the digital net generated by the matrices $C_1,\ldots, C_s$.
\end{definition}
The above $P$ is the same point set as defined
in \cite{LNS} in Section~1.
This follows by just comparing the columns of $C_j$ with the rows of $X_m$. The assumption
that $\#(P)=b^m$ is equivalent to that
$X_1,\ldots,X_m$ give a basis of $P$.

When $R$ is a finite field, the notion of
the dual net is introduced in \cite{np}.
A straight forward generalization to a finite ring is:
\begin{definition}
Let $P\subset \bbX_n^s$ be a left $R$-module and let $X, X' \in P$ with $X = (x_{i,j})$ and $X' = (x'_{i,j})$. We define
$$
P^\vee
:=\{X' \in \bbX_n^s | \sum_{i,j}x_{i,j}'x_{i,j}=0
\mbox{ for any }X \in P \}
\subset \bbX_n^s.
$$
\end{definition}
The expression of the exact $t$-value of $P$
in terms of the minimum NRT-weight of $P^\vee$
is given in \cite{np} in the case where $R$
is a finite field.

To deduce their result
from our character theoretic dual $P^\perp \in \bbX_n^s$,
we need one definition on finite commutative rings, cf. \cite{Wood99}.

\begin{definition}
A finite commutative ring $R$ has a generating character  if
there exists a character $c:R \to T$ as an additive group
such that the pairing obtained by the composition
with the multiplication $\mu$
$$
R \times R \stackrel{\mu}{\to} R \stackrel{c}{\to} T
$$
is a perfect pairing, namely,
$$\theta_c: R \to R^*, x \mapsto c(x\cdot (-))$$
is an isomorphism of abelian groups. We call such a ring $R$ a finite commutative ring with generating character.
\end{definition}
In \cite{Wood99} it was shown that finite commutative rings with a generating character are equivalent to finite commutative Frobenius rings, which in turn are equivalent to finite commutative Gorenstein rings.

It is easy to see that the class of such rings
is closed under taking cartesian products
and subrings (because of the finiteness: a subring $S\subset R$
has $S \to S^*$ as a restriction of $R \to R^*$,
which is an injection and hence bijective).
It includes $\bbZ_b$, since we may then take $c:\bbZ_b \to T$
as an injective group homomorphism. It includes finite fields,
since a finite field $R$ is a
vector space over a prime field $\bbZ_p$,
we may take any
nontrivial $\mathbb{Z}_p$-linear mapping $R \to \mathbb{Z}_p$ and an
injective group homomorphism $\mathbb{Z}_p \to T$, and then $c$ is obtained as the
composition $c: R \to \mathbb{Z}_p \to T$. Thus,
it is a large class of finite commutative associative
rings, but it does not include all finite rings, see Remark~\ref{rem:counter-example} below.
\begin{lemma}
Suppose that $R$ is a finite commutative associative ring with generating character.
Then the isomorphism
$$
\theta_c: \bbX_n^s \to \bbK_n^s
$$
of abelian groups induced by $\theta_c$ componentwise,
induces an isomorphism
$$
P^\vee \to P^\perp.
$$
\end{lemma}
\begin{proof}
For $X'\in \bbX_n^s$,
$\theta_c(X')=(\theta_c(x'_{i,j}\cdot(-))) \in \bbK_n^s$.
The condition $X \in P^\vee$ is equivalent to
$$
\sum_{i,j} x'_{i,j}x_{i,j}=0
\mbox{ for any } X \in P.
$$
The condition $X \in P^\perp$ is,
through the identification $R \to R^*$ above,
$$
c(\sum_{i,j} x'_{i,j}x_{i,j})=1
\mbox{ for any } X \in P.
$$
Thus, $P^\vee \subset P^\perp$ is automatic.
For the converse inclusion, we need the fact
that $P$ is an $R$-module.
The set
$$
S:= \left\{\sum_{i,j} x'_{i,j}x_{i,j} |
X \in P \right\} \subset R
$$
is then an $R$-module (by commutativity of $R$).
The condition $X'\in P^\perp$ is
equivalent to $c(S)=1$. This implies that
any element $s \in S$ has the property $\theta_c(rs)=1$
for all $r$. By the assumption of the character,
this implies $s=0$, hence $S=0$, and $X' \in P^\vee$.
\end{proof}
\begin{corollary}
Let $R$ be a finite commutative ring with generating character.
The result in \cite{np} based on
the ring theoretic dual net $P^\vee$
follows from our character theoretic result Theorem~\ref{th:dual-net-minwt}.
\end{corollary}

\begin{remark}\label{rem:counter-example}
Let $R$ be the finite commutative associative ring
$\bbZ_p[x,y]/(x^2,xy,y^2)$.
Then the $\bbZ_p$-linear span $P:=<x> \subset R$
is an $R$-module of cardinality $p$ and
satisfies $P^\vee=<x,y>$, whose cardinality is $p^2$.
On the other hand, $P^\perp$ has the same
cardinality $p$ as $P$. Thus, $R$ does not have a generating character
and the cardinality of $P^\vee$ is strictly larger than the cardinality of $P^\perp$.
\end{remark}

We see that if $R$ is finite
commutative associative ring with generating character, then
the two notions of duality coincide, and
the minimum NRT-weight of $P^\vee$ gives the strict
$t$-value.

Suppose that $R$ is a finite
commutative associative ring, not assuming
that it has a generating character. Still, if $P$ is a free $R$-module,
the same result for the minimum NRT-weight holds, by the following proposition.
\begin{proposition}
Under the above assumption on $R$,
let $P \subset \bbX_n^s$ be a free $R$-module.
Then, for any $Y$,
the composition $P \subset \bbX_n^s \to \bbX_n^s(Y)$
is surjective if and only if
$P^\vee \cap \bbX_n^s(Y)=0$, where
$\bbX_n^s(Y):=\{X=(x_{i,j}) | x_{i,j}=0 \mbox{ for all $(i,j) \notin Y$}
 \}
\subset \bbX_n^s$.
\end{proposition}
\begin{proof}
For an $R$-module $V$,
let us write $\check{R}:=\Hom_R(V, R)$.
It is obvious that if
$f: R^n \to R^m$ is surjective, then
$\check{f}:\check{R^m} \to \check{R^n}$
is injective. The converse, $f$ injective
does not necessarily imply $\check{f}$ surjective
for a general ring, but it is true for
a finite ring $R$ (see the next lemma).

This makes it possible to replace $P^*$ with $\check{P}$,
then $P^\perp$ with $P^\vee$, in the proof
of Lemma~\ref{lem:duality-in-perp}, if $P$ is
a free $R$-module.
\end{proof}
\begin{corollary}
Suppose that $P$ is a free $R$-module.
Then, the statement of Theorem~\ref{th:dual-net-minwt}
is true if $P^\perp$ is replaced with $P^\vee$.
\end{corollary}

The following result is a deeper result than it looks like.
\begin{lemma}
Let $R$ be a commutative associative ring of Krull dimension zero
(this condition is satisfied if $R$ is a finite commutative associative
ring). Then, if a morphism $f:R^n \to R^m$ of the $R$-module is
injective, then $\check{f}$ is surjective.
\end{lemma}
\begin{proof}
For such a ring, it is proved in \cite[Proposition~2.6]{NN91}
that for any $m \times n$ matrix $A$ of coefficients in $R$,
if the multiplication of $A$ from left to $R^n$ is injective,
then one can extend $A$ to an invertible
matrix by adding $m-n$ columns. In particular,
the multiplication of $A$ from the right to a
horizontal vector $R^n$ is surjective.
This means that the multiplication of the transpose
${}^\top A$ from the left is surjective, which is
the representation matrix of $\check{f}$,
hence $\check{f}$ is surjective.
\end{proof}

\section{Acknowledgements}

The authors would like to thank Harald Niederreiter for helpful
discussions and comments on the manuscript. The first author is
supported by an QE2 Fellowship of the Australian Research
Council.
The second author
is supported by
JSPS/MEXT Grant-in-Aid for Scientific Research
No.24654019, No.23244002, No.21654017.
J. D. is grateful for the hospitality of Prof. Matsumoto while
visiting
the University of Tokyo where most of this research was carried out.

\end{document}